\documentclass
{siamart220329}
\usepackage{braket,amsfonts}
\usepackage{amssymb}
\usepackage{mathrsfs}
\usepackage{array}
\usepackage{stmaryrd}
\usepackage{enumitem}
\usepackage{cases}
\newsiamthm{lem}{Lemma}
\newsiamthm{prop}{Proposition}
\newsiamthm{thm}{Theorem}
\newsiamremark{rmk}{Remark}
\newsiamthm{ass}{Assumption}
\newsiamthm{defn}{Definition}
\newsiamthm{pro}{Problem}
\newsiamthm{cor}{Corollary}
\newsiamthm{example}{Example}

\renewcommand{\d}{\mathrm{d}}

\renewcommand{\theequation}{\thesection.\arabic{equation}}
\makeatletter\@addtoreset{equation}{section} \makeatother

\allowdisplaybreaks
\begin{document}
	\title{The Ergodic Linear-Quadratic Optimal Control Problems for Stochastic Mean-Field Systems with Periodic Coefficients\thanks{Submitted to the editors \today
		\funding{This paper is supported by National Key R\&D Program of China (No.2022YFA1006101), National Natural Science Foundation of China (No.12371445) and the Science and Technology Commission of Shanghai Municipality (No.22ZR1407600).}}}
\author{Jiacheng Wu
	\thanks{School of Mathematical Sciences, Fudan University, Shanghai 200433, China (\email{22110180042@m.fudan.edu.cn}).}
	\and Qi Zhang
	\thanks{Corresponding author. School of Mathematical Sciences, Fudan University, Shanghai 200433, China and Laboratory of Mathematics for Nonlinear Science, Fudan University, Shanghai 200433, China (\email{qzh@fudan.edu.cn}).}
}
		\maketitle
\begin{abstract}
In this paper, we concern with the ergodic linear-quadratic closed-loop optimal control problems, in which the state equation is the mean-field stochastic differential equation with periodic coefficients. We first study the asymptotic behavior of the solution to the state equation and get a family of periodic measures depending on time variables within a period from the convergence of transition probabilities. Then, with the help of periodic measures and periodic Riccati equations, we transform the ergodic cost functional on infinite horizon into an equivalent cost functional on a single periodic interval without limit, and
present the closed-loop optimal controls for our concerned control system. Finally, an example is given to demonstrate the applications of our theoretical results.
\end{abstract}

\section{Introduction}

The study of deterministic linear-quadratic (LQ) optimal control problem has a long history which is traced back to Kalman \cite{ref9} in 1960. Later, the stochastic LQ optimal control problem with deterministic coefficients also appeared in  Wonham \cite{ref10} in 1968. But the stochastic LQ optimal control problem with random coefficients , put forward by  Bismut \cite{ref11} in 1976,
is quite different, for which the Riccati equation is a non-Lipschitz backward stochastic differential equation. Bismut indicated that the solvability of Riccati equation is ``a challenging task" which was solved by Tang \cite{ta} in 2003.

In the past half century, LQ problem has been widely studied and extended to various types. Among these extended researches on LQ problem, the mean-field state equation is a special and important one. The mean-field system simplifies the complex interactions in large-scale systems by replacing intricate microscopic dynamics with a macroscopic ``average" description. In recent years, the mean-field system has become a hot topic in theories and practices across many fields like physics, mathematics, economics, biology, and machine learning, etc.

Yong \cite{ref22}, published in 2013, first studied the LQ optimal control problem for stochastic mean-field system on finite horizon. It needs to point out that, the mean-field term causes nontrivial troubles in the classical method by the Riccati equation. The key point to cover this gap is to take the expectation and the deviation of state variable as dual variables. Moreover, similar to the LQ problem for classic stochastic system, the LQ problem for stochastic mean-field system on finite horizon is uniquely open-loop solvable, and its open-loop optimal control has a closed-loop form which is well presented by the dual Riccati equations and dual variables.
After that, there were many works on this topic. Huang, Li and Yong \cite{ref23} extended this result to infinite horizon (but not in a ergodic form) in 2015; Yong \cite{ref25} investigated the time-consistent properties of the optimal control in 2017; Sun \cite{ref26} studied the indefinite case and open-loop solvability of such kind of system in the same year, to name but a few.

On the other hand, the stochastic ergodic control problem has a history as long as LQ control problem. A pioneer work \cite{wo} in 1967 by Wonham first addressed the ergodic control problem for a stochastic state equation without control in the diffusion term. Since then this control system has garnered significant attentions from the research community.
We only recall some results on ergodic LQ control problems. In 2009 two works \cite{ref14,ref15} by Guatteri and Tessitore were concerned with the ergodic LQ optimal control problem with stochastic coefficients in 2009;
Mei, Wei and Yong \cite{ref21} studied the ergodic LQ optimal control problem with constant coefficients in 2021, in which the invariant measure plays a role in the depict of optimal control and the cost functional on infinite horizon degenerates to an integration with respect to invariant measure. Obviously, systems with constant coefficients inherently exhibit stationary and periodic properties.

As for the stochastic optimal control problems with periodic coefficients, there are only few results. Sun-Yong \cite{ref30} investigated the exponential turnpike property of LQ optimal control problem in 2024, which demonstrated an approximating of the solution of LQ optimal control problem  to the solution of a dynamic optimization problem. Ma \cite{ref31} studied the ergodic optimal control problem for stochastic system with Lipschitz and periodic coefficients and derived the pathwise periodic property for a type of HJB equation by the dynamic programming method in 2024.

Different from existing results, we study the ergodic LQ closed-loop optimal control problems for mean-field stochastic differential equation with periodic coefficients in this paper. Inspired by the periodic measure appearing in Feng and Zhao \cite{fe-zh} and Sun and Yong \cite{ref30}, we get a family of periodic measures, depending on time variable within a period, from the state equation with periodic coefficients. With the help of periodic measures, an equivalent form of ergodic LQ cost functional on a single periodic interval without limit is obtained. Then the optimal solutions of ergodic LQ optimal control problem are presented by periodic measures. As far as we know, there is no result to connect optimal solutions of stochastic control problems on infinite horizon with the periodic measures.

The rest of this paper is organized as follows. In Section 2, we introduce the concerned ergodic LQ optimal control problem and prove some useful estimates on stochastic mean-field system. Section 3 is devoted to the long time asymptotic behavior of the mean-field stochastic differential equation and its periodic measures are proved.
Then the solutions of the ergodic LQ optimal control problem are presented based on the periodic measures, periodic Riccati equations and other related ordinary differential equation (ODE) in Section 4. Finally, we give an example for the concerned control system to demonstrate our theoretical results in Section 5.

\section{Notation and Control Problem}
\label{sec:main}

Let $(\Omega,\mathscr{F},\mathbb{P})$ be a complete probability space, on which a standard $1$-dimensional Brownian motion $\left\{W_t,t\ge 0\right\}$ is defined with $\mathbb{F}=\left\{\mathscr{F}_t\right\}_{t\ge0}$ be the natural filtration of $W$ augmented by the $\mathbb{P}$-null sets in $\mathscr{F}$.

For any $T\in(0,\infty]$ and Euclidean space $E$, define\\
{\small		$\bullet$ $L_{\mathscr{F}}^2(\Omega;E)=\{\xi:\Omega\rightarrow E\Big|\xi\ \text{is}\ \mathscr{F}\text{-measurable},\mathbb{E}|\xi|^2<\infty\}$,\\
		$\bullet$ $L_{\mathbb{F}}^2(0,T;E)=\{X:\left[0,T\right]\times\Omega\rightarrow E\Big|X\ \text{is}\ \mathbb{F}\text{-progressively measurable},\ \mathbb{E}\int_0^{T}|X_t|^2\d t<\infty\}$,\\
$\bullet$ $L_{\mathbb{F}}^{2,loc}(0,\infty;E)= \bigcap\limits_{T>0}L_{\mathbb{F}}^2(0,T;E)$.}

Denote by $\mathbb{S}^n$ the space of all $n\times n$ symmetric real matrices, by $\mathbb{S}_{+}^n$ the space of all $n\times n$ positive definite real matrices and by $\bar{\mathbb{S}}_{+}^n$ the space of all $n\times n$ positive semi-definite real matrices.
For a measurable function $f:[0,\infty)\rightarrow\mathbb{S}^n$, $f$ is called uniformly positive definite if there exists an $\alpha>0$ such that
\begin{equation*}
	\begin{aligned}
		(f_t-\alpha I_n)\in\bar{\mathbb{S}}_{+}^n\ {\rm for}\ {\rm all}\ t\in[0,\infty).
	\end{aligned}
	\nonumber
\end{equation*}

For a $\mathbb{R}^m$-valued control process $u$ and a given $x\in\mathbb{R}^n$, the linear mean-field state equation we study is as follow
\begin{equation}\label{wz3}
\left\{\begin{aligned}
	\d X_t=&(A_tX_t+\overline{A}_t\mathbb{E}X_t+B_tu_t+\overline{B}_t\mathbb{E}u_t+b_t)\d t\\
	&+(C_tX_t+\overline{C}_t\mathbb{E}X_t+D_tu_t+\overline{D}_t\mathbb{E}u_t+\sigma_t)\d W_t,\ \ \ t\ge 0,\\
	X_0=&x,
\end{aligned}\right.
\end{equation}
where $A,\overline{A},C,\overline{C}:[0,\infty)\rightarrow\mathbb{R}^{n\times n}$, $B,\overline{B},D,\overline{D}:[0,\infty)\rightarrow\mathbb{R}^{n\times m}$, $b:[0,\infty)\rightarrow\mathbb{R}^n$ and $\sigma:[0,\infty)\rightarrow\mathbb{R}^n$ are measurable functions.

For $T>0$, define
\begin{equation}\label{wz4}
	\begin{aligned}
		J_T(x,u)=\mathbb{E}\int_0^TF(t,X_t,\mathbb{E}X_t,u_t,\mathbb{E}u_t)\d t,
	\end{aligned}
\end{equation}
where
\begin{equation*}
	\begin{aligned}
	F(t,X_t,\mathbb{E}X_t,u_t,\mathbb{E}u_t)
	=&\left\langle Q_tX_t,X_t\right\rangle+2\left\langle S_tX_t,u_t\right\rangle+\left\langle R_tu_t,u_t\right\rangle+2\left\langle q_t,X_t\right\rangle+2\left\langle \rho_t,u_t\right\rangle\\
	&+\left\langle \overline{Q}_t\mathbb{E}X_t,\mathbb{E}X_t\right\rangle+2\left\langle \overline{S}_t\mathbb{E}X_t,\mathbb{E}u_t\right\rangle+\left\langle \overline{R}_t\mathbb{E}u_t,\mathbb{E}u_t\right\rangle,
	\end{aligned}
\end{equation*}
and $Q,\overline{Q}:[0,\infty)\rightarrow\mathbb{S}^{n}$, $S,\overline{S}:[0,\infty)\rightarrow\mathbb{R}^{m\times n}$, $R,\overline{R}:[0,\infty)\rightarrow\mathbb{S}^{m}$, $q:[0,\infty)\rightarrow\mathbb{R}^n$, $\rho:[0,\infty)\rightarrow\mathbb{R}^m$ are measurable functions, $\left\langle\cdot,\cdot\right\rangle$ denotes the inner product of two matrices in the sense that $\left\langle M,N\right\rangle=\text{tr}\left(M^\top N\right)$ for any $M,N\in\mathbb{R}^{m\times n}$.
The cost functional is as follow
\begin{equation*}
	\begin{aligned}
		\mathcal{E}(x,u)=\varliminf\limits_{T \to \infty}\frac{1}{T}J_T(x,u).
	\end{aligned}
\end{equation*}

We study the closed-loop control for the above state equation and cost functional. So for $t\in[0,T]$, the control-state pair $(u^{0,x,\Theta,\overline{\Theta},v}, X^{0,x,\Theta,\overline{\Theta},v})$ has a form
\begin{equation}\label{wz1}
	\begin{aligned}
u^{0,x,\Theta,\overline{\Theta},v}_t=\Theta_tX^{0,x,\Theta,\overline{\Theta},v}_t
+\overline{\Theta}_t\mathbb{E}X^{0,x,\Theta,\overline{\Theta},v}_t+v_t
	\end{aligned}
\end{equation}
and
{\small\begin{equation}\label{wz2}
	\left\{\begin{aligned}
		\d X^{0,x,\Theta,\overline{\Theta},v}_t=&\Big((A_t+B_t\Theta_t)X^{0,x,\Theta,\overline{\Theta},v}_t
+(\overline{A}_t+B_t\overline{\Theta}_t+\overline{B}_t\Theta_t
+\overline{B}_t\overline{\Theta}_t)\mathbb{E}X^{0,x,\Theta,
\overline{\Theta},v}_t\\
		&+(B_t+\overline{B}_t)v_t+b_t\Big)\d t\\
		&+\Big((C_t+D_t\Theta_t)X^{0,x,\Theta,\overline{\Theta},v}_t+(\overline{C}_t+D_t\overline{\Theta}_t
+\overline{D}_t\Theta_t+\overline{D}_t\overline{\Theta}_t)\mathbb{E}X^{0,x,\Theta,\overline{\Theta},v}_t\\
		&+(D_t+\overline{D}_t)v_t+\sigma_t\Big)\d W_t,\ \ \ t\ge 0,\\
		X^{0,x,\Theta,\overline{\Theta},v}_0=&x,
	\end{aligned}\right.
\end{equation}}
where $(\Theta,\overline{\Theta},v):[0,T]\rightarrow\mathbb{R}^{n\times n}\times\mathbb{R}^{n\times n}\times\mathbb{R}^{n}$ is a measurable triple function. Similarly, we denote by $X^{t_0,x,\Theta,\overline{\Theta},v}_t$, $t\geq t_0$, the solution to (\ref{wz2}) with the initial time $0$ replaced by given $t_0\geq0$. By defining $X^{t_0,x,\Theta,\overline{\Theta},v}_t\equiv x$, $0\leq t\leq t_0$, we get a solution $X^{t_0,x,\Theta,\overline{\Theta},v}$ to (\ref{wz2}) on $[0,\infty)$.

Denote by $\mathbb{U}$ the set of measurable functions $(\Theta,\overline{\Theta},v)$ which guarantees the control-state pair $(u^{0,x,\Theta,\overline{\Theta},v}, X^{0,x,\Theta,\overline{\Theta},v})\in L_{\mathbb{F}}^{2,loc}(0,\infty;\mathbb{R}^{m})\times L_{\mathbb{F}}^{2,loc}(0,\infty;\mathbb{R}^{n})$ and will be further specified. Then we define the admissible set
\begin{equation*}
	\begin{aligned}		\mathcal{U}=\Big\{u^{0,x,\Theta,\overline{\Theta},v}\in L_{\mathbb{F}}^{2,loc}(0,\infty;\mathbb{R}^{m})\Big|u^{0,x,\Theta,\overline{\Theta},v}\ {\rm satisfies}\ (\ref{wz1})\ {\rm with}\ (\Theta,\overline{\Theta},v)\in\mathbb{U}\Big\}.
	\end{aligned}
\end{equation*}

Now the ergodic control problem for the stochastic LQ mean-field system can be formulated.

\textbf{Problem (CL-MFLQE).} Find a closed-loop control $u^{0,x,\Theta^*,\overline{\Theta}^*,v^*}\in\mathcal{U}$ such that
\begin{equation*}
	\begin{aligned}
		\mathcal{E}(x,u^{0,x,\Theta^*,\overline{\Theta}^*,v^*})=
\inf\limits_{u^{0,x,\Theta,\overline{\Theta},v}\in\mathcal{U}}\mathcal{E}(x,u^{0,x,\Theta,\overline{\Theta},v}).
	\end{aligned}
\end{equation*}

\textbf{Problem (CL-MFLQE)} is called closed-loop solvable, if there exists an optimal control $u^{0,x,\Theta^*,\overline{\Theta}^*,v^*}\in\mathcal{U}$, which together with its corresponding optimal state $X^{0,x,\Theta^*,\overline{\Theta}^*,v^*}$ constitutes an optimal control-state pair. The function $$V(x)=\inf\limits_{u^{0,x,\Theta,\overline{\Theta},v}\in\mathcal{U}}\mathcal{E}(x,u^{0,x,\Theta,\overline{\Theta},v})$$ is called the value function of \textbf{Problem (CL-MFLQE)}.

Before the assumptions on the coefficients are presented, we clarify some basic notation.
\begin{defn}
Given $\tau>0$, a measurable function $f:[0,\infty)\rightarrow E$ is called $\tau$-periodic if
\begin{equation*}
	\begin{aligned}
		f_{t+\tau}=f_t\ \ \ {\rm for}\ {\rm any}\ t\geq0.
	\end{aligned}
	\nonumber
\end{equation*}
\end{defn}
Then we define some spaces of $\tau$-periodic functions\\
$\bullet$ $\mathscr{B}_\tau(E)=\Big\{f:[0,\infty)\rightarrow E\Big|f\text{ is Lebesgue essentially bounded and }\tau\text{-periodic}\Big\}$,\\
$\bullet$ $\mathscr{C}_\tau(E)=\Big\{f:[0,\infty)\rightarrow E\Big|f\text{ is continuous and }\tau\text{-periodic}\Big\}$,\\
$\bullet$ $\mathscr{D}_\tau(E)=\Big\{f:[0,\infty)\rightarrow E\Big|f\text{ and its derivative }\dot{f} \text{ are both in }\mathscr{B}_\tau(E)\Big\}$.\\
We use the brief notation $\mathscr{B}_\tau,\mathscr{C}_\tau,\mathscr{D}_\tau$ if the value space $E$ is clear in above definitions. To avoid heavy notation, for $t\geq0$, we also denote
\begin{equation*}
	\begin{aligned}
		&\widehat{A}_t=A_t+\overline{A}_t,\ \widehat{B}_t=B_t+\overline{B}_t,\ \widehat{C}_t=C_t+\overline{C}_t,\ \widehat{D}_t=D_t+\overline{D}_t,\\
		&\widehat{Q}_t=Q_t+\overline{Q}_t,\ \widehat{S}_t=S_t+\overline{S}_t,\ \widehat{R}_t=R_t+\overline{R}_t,\ \widehat{\Theta}_t=\Theta_t+\overline{\Theta}_t.
	\end{aligned}
	\nonumber
\end{equation*}
To deal with the infinite horizon state equation, we introduce the mean-square exponentially stabilizability.
\begin{definition}\label{wz8} (Definition 3.7 in \cite{ref30}) For $\tau>0$, let $A,C:[0,\infty)\rightarrow\mathbb{R}^{n\times n}$ and $B,D:[0,\infty)\rightarrow\mathbb{R}^{n\times m}$ be $\tau$-periodic matrix-valued measurable functions. The 4-tuple of coefficients $[A,C;B,D]$ is $\tau$-periodic mean-square exponentially stabilizable if there exist $\Theta\in \mathscr{B}_\tau(\mathbb{R}^{m\times n})$ and $M,\lambda>0$ such that the homogeneous closed-loop system
{\small\begin{equation}\label{homogeneous system}
\left\{\begin{aligned}
	\d\Phi_t&=(A_t+B_t\Theta_t)\Phi_t\d t+(C_t+D_t\Theta_t)\Phi_t\d W_t,\ \ \ t\ge 0,\\
	\Phi_0&=I_n,
\end{aligned}\right.
\end{equation}}
has a unique solution $\Phi\in L_{\mathbb{F}}^{2}(0,\infty;\mathbb{R}^{n\times n})$ and for any $t\ge 0$,
$\mathbb{E}|\Phi_t|^2\le Me^{-\lambda t}$.
Here $\Theta$ is called a $\tau$-periodic stabilizer of $[A,C;B,D]$
and the set of all $\tau$-periodic stabilizers of $[A,C;B,D]$ is denoted by $\mathscr{S}_\tau[A,C;B,D]$.
\end{definition}

Then we introduce the assumptions on coefficients in (\ref{wz3}) and (\ref{wz4}).
\begin{ass}
{\rm\textbf{(A1)}} For a given $\tau>0$, $A,\overline{A},B,\overline{B},C,\overline{C},D,\overline{D}, b,\sigma, Q,\overline{Q},S,\\\overline{S},R,\overline{R},q,\rho\in\mathscr{B}_\tau$.\\
{\rm\textbf{(A2)}} $R,\widehat{R},Q-S^\top R^{-1}S,\widehat{Q}-\widehat{S}^\top\widehat{R}^{-1}\widehat{S}$ are uniformly positive definite.\\
{\rm\textbf{(A3)}} Both $[A,C;B,D]$ and $[\widehat{A},0;\widehat{B},0]$ are $\tau$-periodic mean-square exponentially stabilizable.
\end{ass}

Based on \textbf{(A1)} and \textbf{(A3)} the coefficients in (\ref{wz3}) and (\ref{wz4}) satisfy, we specify the $\mathbb{U}$ in the admissible control set, i.e.
{\small\begin{equation*}
	\begin{aligned}		\mathbb{U}=\Big\{(\Theta,\overline{\Theta},v)\Big|\Theta\in\mathscr{S}_\tau[A,C;B,D],
\widehat{\Theta}\in\mathscr{S}_\tau[\widehat{A},0;\widehat{B},0],v\in\mathscr{B}_\tau(\mathbb{R}^{m})\Big\}.
	\end{aligned}
\end{equation*}}

Obviously, the above $\mathbb{U}$ guarantees $(u^{0,x,\Theta,\overline{\Theta},v}, X^{0,x,\Theta,\overline{\Theta},v})\in L_{\mathbb{F}}^{2,loc}(0,\infty;\mathbb{R}^{m})\times L_{\mathbb{F}}^{2,loc}(0,\infty;\mathbb{R}^{n})$ and the well-posedness of ergodic cost functional, i.e. for any $T>0$,
{\small\begin{equation*}
	\begin{aligned}		\frac{1}{T}\mathbb{E}\int_0^TF(t,X^{0,x,\Theta,\overline{\Theta},v}_t,\mathbb{E}X^{0,x,\Theta,\overline{\Theta},v}_t,u^{0,x,\Theta,\overline{\Theta},v}_t,\mathbb{E}u^{0,x,\Theta,\overline{\Theta},v}_t)\d t<\infty.
	\end{aligned}
\end{equation*}}

Based on the periodic property, we have an estimation for the solution to the homogeneous closed-loop system (\ref{homogeneous system}).
\begin{lemma}\label{wz5} (Corollary 3.4 in \cite{ref30}) For a given $\tau>0$, assume $A,B,C,D\in\mathscr{B}_\tau$ and that $[A,C;B,D]$ is $\tau$-periodic mean-square exponentially stabilizable. Then there exist $M,\lambda>0$ depending only on given parameters such that for any $0\le s\le t$, the solution $\Phi$ to \eqref{homogeneous system} satisfies
$\mathbb{E}\big|\Phi_t\Phi_s^{-1}\big|^2\le Me^{-\lambda(t-s)}$.
\end{lemma}

With the help of Lemma \ref{wz5}, we have the following estimates for the closed-loop state equation.
\begin{proposition}\label{closed-loop estimate} For a given $\tau>0$, assume $A,\overline{A},B,\overline{B},C,\overline{C},D,\overline{D}, b,\sigma\in\mathscr{B}_\tau$ and {\rm\textbf{(A3)}}. The closed-loop state equation (\ref{wz2}) with initial time $t_0\geq0$, initial state $x\in\mathbb{R}^n$ and $(\Theta,\overline{\Theta},v)\in\mathbb{U}$ has a unique solution $X^{t_0,x,\Theta,\overline{\Theta},v}\in L_{\mathbb{F}}^{2,loc}(0,\infty;\mathbb{R}^{n})$, and there exist $M,\lambda>0$ depending only on given parameters such that for any $t\ge t_0$,
{\small\begin{equation}\label{wz6}
	\begin{aligned}
		&\big|\mathbb{E}X^{t_0,x,\Theta,\overline{\Theta},v}_t\big|^2\le M(1+e^{-\lambda(t-t_0)}|x|^2),\\ &\mathbb{E}\big|X^{t_0,x,\Theta,\overline{\Theta},v}_t-\mathbb{E}X^{t_0,x,\Theta,\overline{\Theta},v}_t\big|^2\le M(1+e^{-\frac{\lambda}{2}(t-t_0)}|x|^2).
	\end{aligned}
\end{equation}}
For another solution $X^{t_0,x',\Theta',\overline{\Theta}',v'}$ with initial state $x'\in\mathbb{R}^n$ and $(\Theta',\overline{\Theta}',v')\in\mathbb{U}$
satisfying (\ref{wz6}) with $M',\lambda'>0$ as $X^{t_0,x,\Theta,\overline{\Theta},v}$ with $M,\lambda>0$, there exists $\hat{M}>0$ depending only on given parameters such that for any $t\ge t_0$,
{\small\begin{equation}\label{wz7}
	\begin{aligned}
& \big|\mathbb{E}[X^{t_0,x,\Theta,\overline{\Theta},v}_t-X^{t_0,x',\Theta',\overline{\Theta}',v'}_t]\big|^2 \\ \le& \hat{M}e^{-{\lambda\vee \lambda'}(t-t_0)}|x-x'|^2\\
&+\hat{M}\int_{t_0}^t\Big((1+|x|^2+|x'|^2)(|\Theta_s-\Theta'_s|^2
+|\overline{\Theta}_s-\overline{\Theta}'_s|^2)+|v_s-v'_s|^2\Big)\d s.
    \end{aligned}
\end{equation}}
Moreover, if $\Theta_\cdot=\Theta'_\cdot$, it yields that
{\small\begin{equation}\label{wz10}
	\begin{aligned} &\mathbb{E}\big|X^{t_0,x,\Theta,\overline{\Theta},v}_t-X^{t_0,x',\Theta',\overline{\Theta}',v'}_t
-\mathbb{E}[X^{t_0,x,\Theta,\overline{\Theta},v}_t-X^{t_0,x',\Theta',\overline{\Theta}',v'}_t]\big|^2\\
	\le &Me^{-\frac{\lambda\vee \lambda'}{2}(t-t_0)}|x-x'|^2+M\int_{t_0}^t\Big((1+|x|^2+|x'|^2)|\overline{\Theta}_s-\overline{\Theta}'_s|^2
+|v_s-v'_s|^2\Big)\d s.
	\end{aligned}
\end{equation}}
\end{proposition}
\begin{proof}
The existence and uniqueness of solution to (\ref{wz2}) follows immediately from the essential boundedness of all coefficients. To prove (\ref{wz6}), without loss of generality, we take $t_0=0$ and use the simple notation $X_\cdot$ and $X'_\cdot$ for $X^{0,x,\Theta,\overline{\Theta},v}_\cdot$ and $X^{0,x,\Theta',\overline{\Theta}',v'}_\cdot$, respectively.

By {\rm\textbf{(A3)}} we have two matrix stochastic differential equations (SDEs) (\ref{homogeneous system}) and
{\small\begin{equation*}
\left\{\begin{aligned}
	\d\Psi_t&=(\widehat{A}_t+\widehat{B}_t\widehat{\Theta}_t)\Psi_t\d t,\ \ \ t\ge 0,\\
	\Psi_0&=I_n.
\end{aligned}\right.
\nonumber
\end{equation*}}
Also there exist $M,\lambda>0$ such that $|\Psi_t|^2\le Me^{-\lambda t}$.
For $t\geq0$, denote
$$Y_t=\mathbb{E}X_t,\ Y'_t=\mathbb{E}X'_t,\ Z_t=X_t-Y_t,\ Z'_t=X'_t-Y'_t.$$
Then $Y$ and $Z$ satisfy
\begin{equation*}
\left\{\begin{aligned}
	\d Y_t&=\Big((\widehat{A}_t+\widehat{B}_t\widehat{\Theta}_t)Y_t+\widehat{B}_tv_t+b_t\Big)\d t,\\
	Y_0&=x
\end{aligned}\right.
\end{equation*}
and
\begin{equation*}
	\left\{\begin{aligned}
	\d Z_t=&(A_t+B_t\Theta_t)Z_t\d t+\Big((C_t+D_t\Theta_t)Z_t+(\widehat{C}_t+\widehat{D}_t\widehat{\Theta}_t)Y_t+\widehat{D}_tv_t+\sigma_t\Big)\d W_t,\\
	Z_0=&0,
\end{aligned}\right.
\end{equation*}
respectively. Hence
\begin{equation*}
	\begin{aligned}
		Y_t=\Psi_tx+\int_0^t\Psi_t\Psi_s^{-1}\left(\widehat{B}_sv_s+b_s\right)\d s.
	\end{aligned}
\end{equation*}
Then, by Lemma \ref{wz5} for $\Psi$, we have
\begin{equation*}
	\begin{split}		|Y_t|^2\le&K\left|\Psi_tx\right|^2+K\left|\int_0^t\Psi_t\Psi_s^{-1}\left(\widehat{B}_sv_s+b_s\right)\d s\right|^2\\
		\le&Ke^{-\lambda t}|x|^2+\left|\int_0^tKe^{-\frac{\lambda}{2}(t-s)}\d s\right|^2
		\le K(1+e^{-\lambda t}|x|^2).
	\end{split}
\end{equation*}
Here and in the rest of this paper, $K$ is a generic constant which may change from line to line.
On the other hand, by applying It\^{o} formula to $\Phi_t^{-1}Z_t$ we have
\begin{equation*}\label{wz9}
	\begin{aligned}
		Z_t
		=&\Phi_t\int_0^t\d (\Phi_s^{-1}Z_s)\\
		=&\Phi_t\int_0^t\left(-\Phi_s^{-1}(C_s+D_s\Theta_s)\left((\widehat{C}_s+\widehat{D}_s\widehat{\Theta}_s)Y_s+\widehat{D}_sv_s+\sigma_s\right)\right)\d s\\
		&+\int_0^t\left(\Phi_s^{-1}\left((\widehat{C}_s+\widehat{D}_s\widehat{\Theta}_s)Y_s+\widehat{D}_sv_s+\sigma_s\right)\right)\d W_s\\
		=&\int_0^t\left(-\Phi_t\Phi_s^{-1}(C_s+D_s\Theta_s)(\widehat{C}_s+\widehat{D}_s\widehat{\Theta}_s)Y_s\right)\d s\\
		&+\int_0^t\left(\Phi_t\Phi_s^{-1}(\widehat{C}_s+\widehat{D}_s\widehat{\Theta}_s)Y_s\right)\d W_s\\
		&+\int_0^t\left(-\Phi_t\Phi_s^{-1}(C_s+D_s\Theta_s)(\widehat{D}_sv_s+\sigma_s)\right)\d s\\
		&+\int_0^t\left(\Phi_t\Phi_s^{-1}(\widehat{D}_sv_s+\sigma_s)\right)\d W_s.
	\end{aligned}
\end{equation*}
Hence
\begin{equation*}
	\begin{aligned}
		&\mathbb{E}|Z_t|^2\\		
		\le&K\mathbb{E}\int_0^te^{-\frac{\lambda}{2}(t-s)}\d s\int_0^t e^{\frac{\lambda}{2}(t-s)} \left|\Phi_t\Phi_s^{-1}(C_s+D_s\Theta_s)(\widehat{C}_s+\widehat{D}_s\widehat{\Theta}_s)Y_s\right|^2\d s\\
		&+K\mathbb{E} \int_0^t\left|\Phi_t\Phi_s^{-1}(\widehat{C}_s+\widehat{D}_s\widehat{\Theta}_s)Y_s\right|^2\d s\\
		&+K\mathbb{E}\int_0^te^{-\frac{\lambda}{2}(t-s)}\d s\int_0^t e^{\frac{\lambda}{2}(t-s)}\left|\Phi_t\Phi_s^{-1}(C_s+D_s\Theta_s)(\widehat{D}_sv_s+\sigma_s)\right|^2\d s\\
		&+K\mathbb{E}\int_0^t\left|\Phi_t\Phi_s^{-1}(\widehat{D}_sv_s+\sigma_s)\right|^2\d s\\
		\le&\frac{2}{\lambda}K\int_0^te^{\frac{\lambda}{2}(t-s)}e^{-\lambda(t-s)}e^{-\lambda s}|x|^2\d s+ K\int_0^te^{-\lambda(t-s)}(1+e^{-\lambda s}|x|^2)\d s\\
		&+\frac{2}{\lambda}K\int_0^te^{\frac{\lambda}{2}(t-s)}e^{-\lambda(t-s)}\d s+K\int_0^te^{-\lambda(t-s)}\d s\\
		\le&K(1+e^{-\frac{\lambda}{2}t}|x|^2),
\end{aligned}
\end{equation*}
which implies (\ref{wz6}).

To prove (\ref{wz7}), for $t\geq0$, denote
\begin{equation}
	\begin{aligned}
		&\delta X_t=X_t-X'_t,\ \delta Y_t=Y_t-Y'_t,\ \delta Z_t=Z_t-Z'_t,\
		\delta x=x-x',\\
		&\delta \Theta_t=\Theta_t-\Theta'_t,\ \delta\widehat{\Theta}_t=\widehat{\Theta}_t-\widehat{\Theta}'_t,\ \delta v_t=v_t-v'_t.
	\end{aligned}
	\nonumber
\end{equation}
Assume without loss of generality that $\lambda'\ge\lambda$, we have
\begin{equation*}
\left\{\begin{aligned}
	\d\delta Y_t&=[(\widehat{A}_t+\widehat{B}_t\widehat{\Theta}'_t)\delta Y_t+\widehat{B}_t\delta\widehat{\Theta}_tY_t+\widehat{B}_t\delta v_t]\d t,\\
	\delta Y_0&=\delta x,
\end{aligned}\right.
\end{equation*}
and
\begin{equation*}
\left\{\begin{aligned}
	\d \delta Z_t=&[(A_t+B_t\Theta'_t)\delta Z_t+B_t\delta\Theta_tZ_t]\d t\\
	&+[(C_t+D_t\Theta'_t)\delta Z_t+D_t\delta\Theta_tZ_t+(\widehat{C}_t+\widehat{D}_t\widehat{\Theta}'_t)\delta Y_t\\
	&+\widehat{D}_t\delta\widehat{\Theta}_tY_t+\widehat{D}_t\delta v_t]\d W_t.\\
	\delta Z_0=&0.
\end{aligned}\right.
\end{equation*}
Then
\begin{equation*}
	\begin{split}
		|\delta Y_t|^2\le&K\left|{\Psi'_t}\delta x\right|^2+K\left|\int_0^t{\Psi'_t}({\Psi'_s})^{-1}\left(\widehat{B}_s\delta\widehat{\Theta}_sY_s+\widehat{B}_s\delta v_s\right)\d s\right|^2\\
		\le&Ke^{-\lambda't}|\delta x|^2+K\left|\int_0^te^{-\frac{\lambda'}{2}(t-s)}\left[(1+|x|)|\delta\widehat{\Theta}_s|+|\delta v_s|\right]\d s\right|^2\\
		\le& Ke^{-\lambda't}|\delta x|^2+\frac{K}{\lambda'}\int_0^t\left[(1+|x|^2)|\delta\widehat{\Theta}_s|^2+|\delta v_s|^2\right]\d s.
	\end{split}
\end{equation*}
This is (\ref{wz7}).

Then we turn to (\ref{wz10}). By applying It\^{o} formula to $({\Phi'_t})^{-1}\delta Z_t$, we have
\begin{equation*}
	\begin{aligned}
		\delta Z_t=&{\Phi'_t}({\Phi'_t})^{-1}\delta Z_t\\
		=&\int_0^t{\Phi'_t}({\Phi'_s})^{-1}\Big(\big(B_s-(C_s+D_s\Theta'_s)D_s\big)\delta\Theta_sZ_s\\
&\ \ \ \ \ \ \ \ \ \ \ \ \ \ \ \ \ \ \ \ -(C_s+D_s\Theta'_s)\left((\widehat{C}_s+\widehat{D}_s\widehat{\Theta}'_s)\delta Y_s+\widehat{D}_s\delta\widehat{\Theta}_sY_s+\widehat{D}_s\delta v_s\right)\Big)\d s\\		&+\int_0^t{\Phi'_t}({\Phi'_s})^{-1}\Big(D_s\delta\Theta_sZ_s+(\widehat{C}_s+\widehat{D}_s\widehat{\Theta}'_s)\delta Y_s+\widehat{D}_s\delta\widehat{\Theta}_sY_s+\widehat{D}_s\delta v_s\Big)\d W_s.
	\end{aligned}
	\nonumber
\end{equation*}
Thus if $\delta\Theta_\cdot\equiv0$, we have
{\small\begin{equation*}
	\begin{split}
&\mathbb{E}|\delta Z_t|^2\\			\le&K\mathbb{E}|\int_0^t{\Phi'_t}({\Phi'_s})^{-1}(C_s+D_s\Theta'_s)\Big((\widehat{C}_s+\widehat{D}_s\widehat{\Theta}'_s)\delta Y_s+\widehat{D}_s\delta\widehat{\Theta}_sY_s+\widehat{D}_s\delta v_s\Big)\d s|^2\\		&+K\mathbb{E}|\int_0^t{\Phi'_t}({\Phi'_s})^{-1}\Big((\widehat{C}_s+\widehat{D}_s\widehat{\Theta}'_s)\delta Y_s+\widehat{D}_s\delta\widehat{\Theta}_sY_s+\widehat{D}_s\delta v_s\Big)\d W_s|^2\\
			\le&K\mathbb{E}\int_0^t|{\Phi'_t}({\Phi'_s})^{-1}|^2e^{\frac{\lambda'}{2}(t-s)}\d s\int_0^te^{-\frac{\lambda'}{2}(t-s)}\Big(|\delta Y_s|^2+|\delta\widehat{\Theta}_sY_s|^2+|\delta v_s|^2\Big)\d s\\
			&+K\mathbb{E}\int_0^t|{\Phi'_t}({\Phi'_s})^{-1}|^2\Big(|\delta Y_s|^2+|\delta\widehat{\Theta}_sY_s|^2+|\delta v_s|^2\Big)\d s\\			
			\le&\left(\frac{2}{\lambda'}K+K\right)\int_0^te^{-\frac{\lambda'}{2}(t-s)}\Big(|\delta Y_s|^2+|\delta\widehat{\Theta}_sY_s|^2+|\delta v_s|^2\Big)\d s\\
			\le&K\int_0^te^{-\frac{\lambda'}{2}(t-s)}\Big(e^{-\lambda's}|\delta x|^2+\int_0^s\big((1+|x|^2)|\delta\widehat{\Theta}_r|^2+|\delta v_r|^2\big)\d r\Big)\d s\\
			&+K\int_0^te^{-\frac{\lambda'}{2}(t-s)}\big((1+|x|^2)|\delta\widehat{\Theta}_s|^2+|\delta v_s|^2\big)\d s\\
			\le&\left(\frac{2}{\lambda'}\right)^2Ke^{-\frac{\lambda'}{2}t}|\delta x|^2+\left(\left(\frac{2}{\lambda'}\right)^2K+\frac{2}{\lambda'}K\right)\int_0^t\Big((1+|x|^2)|\delta\widehat{\Theta}_s|^2+|\delta v_s|^2\Big)\d s,
	\end{split}
\end{equation*}}
which puts an end of Proposition \ref{closed-loop estimate}.
\end{proof}

\section{Long time asymptotic behavior of stochastic flow}
To begin with, we recall some definitions and results on Wasserstein distance we use in this section.

Let $\mathscr{L}(\mathbb{R}^d)$ be the Lebesgue $\sigma$-field of $\mathbb{R}^d$, and define a set of probability
\begin{equation*}
	\begin{aligned}
		\mathcal{P}_2(\mathbb{R}^d)=\left\{\nu:\mathscr{L}(\mathbb{R}^d)\rightarrow[0,1]\Big|
\int_{\mathbb{R}^d}|x|^2\nu(\d x)<\infty\right\}.
	\end{aligned}
\end{equation*}
Then $\mathcal{P}_2(\mathbb{R}^d)$ is a complete metric space endowed by 2-Wasserstein distance $w_2$, i.e. for any $\mu,\mu'\in\mathcal{P}_2(\mathbb{R}^d)$,
\begin{equation*}
	\begin{aligned}
		w_2(\mu,\mu')=\inf\Bigg\{\Big(\int_{\mathbb{R}^{2d}}|x-y|^2\nu(\d x,\d y)\Big)^\frac{1}{2}\Bigg|&\nu\in\mathcal{P}_2(\mathbb{R}^{2d}),\ \nu(\d x,\mathbb{R}^d)=\mu(\d x)\\
		&{\rm and}\ \nu(\mathbb{R}^d,\d y)=\mu'(\d y)\Bigg\}.
	\end{aligned}
\end{equation*}

The following two lemmas give some well-known results for 2-Wasserstein distance. The readers can refer to Section 6 in \cite{ref5} for details.
\begin{lemma} For $\xi,\xi'\in L_{\mathscr{F}}^2(\Omega;\mathbb{R}^d)$, $\mathbb{P}_{\xi}$, $\mathbb{P}_{\xi'}$ are their corresponding distributions, respectively. Then
	\begin{equation*}
		\begin{aligned}
			w_2(\mathbb{P}_{\xi},\mathbb{P}_{\xi'})^2\le \mathbb{E}|\xi-\xi'|^2.
		\end{aligned}
	\end{equation*}
\end{lemma}
\begin{lemma}\label{w_2 convergence equivalent condition} 
If $\mu,\mu_{k}\in \mathcal{P}_2(\mathbb{R}^d),k\in\mathbb{N}^+$, the following statements are equivalent:\\
(i) $\lim\limits_{k\rightarrow\infty}w_2(\mu_{k},\mu)=0$;\\
(ii) $\mu_{k}$ weakly converges to $\mu$ and $\lim\limits_{k\rightarrow\infty}\int_{\mathbb{R}^d}|x|^2\mu_{k}(\d x)=\int_{\mathbb{R}^d}|x|^2\mu(\d x)$;\\
(iii) for any quadratic growth $\varphi\in C(\mathbb{R}^d;\mathbb{R})$, i.e. there exists $K\ge0$ such that for any $x\in\mathbb{R}^d$, $|\varphi(x)|\le K(1+|x|^2)$, it holds that $\lim\limits_{k\rightarrow\infty}\int_{\mathbb{R}^d}\varphi(x)\mu_{k}(\d x)=\int_{\mathbb{R}^d}\varphi(x)\mu(\d x)$.
\end{lemma}

For any $0\le t\le s$, $x\in\mathbb{R}^n$ and $B\in\mathscr{B}_{\mathbb{R}^n}$, we define the transition probability
\begin{equation*}
	\begin{aligned}
		p^{\Theta,\overline{\Theta},v}(t,s,x;B)=\mathbb{P}(X^{t,x,\Theta,\overline{\Theta},v}_s\in B).
	\end{aligned}
\end{equation*}

Due to the continuity of initial state for stochastic flow, it is not surprised to have the following result.
\begin{proposition}\label{Feller property} For a given $\tau>0$, assume $A,\overline{A},B,\overline{B},C,\overline{C},D,\overline{D}, b,\sigma\in\mathscr{B}_\tau$, and $X^{r,x,\Theta,\overline{\Theta},v}$, $r\in[0,\tau)$ and $(\Theta,\overline{\Theta},v)\in\mathbb{U}$, is the solution of closed-loop system \eqref{wz2} with the initial time $r$ and initial state $x$. Then for any $h\in C_b(\mathbb{R}^n)$, $t\ge r$,
\begin{equation*}
	\begin{aligned}
		x\mapsto\int_{\mathbb{R}^n}h(y)p^{\Theta,\overline{\Theta},v}(r,t,x;\d y)=\mathbb{E}[h(X^{r,x,\Theta,\overline{\Theta},v}_t)]
	\end{aligned}
\end{equation*}
is continuous.
\end{proposition}

Also we show that the transition probability is periodic.
\begin{proposition}\label{periodic time shift property} For a given $\tau>0$, assume $A,\overline{A},B,\overline{B},C,\overline{C},D,\overline{D}, b,\sigma\in\mathscr{B}_\tau$. Then for any $r\in[0,\tau)$, $t\ge 0$, $k\in\mathbb{N}$ and $(\Theta,\overline{\Theta},v)\in\mathbb{U}$,
\begin{equation*}
	\begin{aligned}
		p^{\Theta,\overline{\Theta},v}(r,t+r,x;\cdot)=p^{\Theta,\overline{\Theta},v}(k\tau+r,t+k\tau+r,x;\cdot).
	\end{aligned}
\end{equation*}
\end{proposition}
\begin{proof}
For any $s,t\ge0$, set
\begin{equation*}
	\begin{aligned}
		W_t^s=W_{t+s}-W_s.
	\end{aligned}
\end{equation*}
For any $s\ge0$, $W^s_\cdot$ is a Brownian motion with the same distribution to $W_\cdot$.
Notice that for any $r\in[0,\tau)$, $t\ge 0$, $k\in\mathbb{N}$, $X^{r,x,\Theta,\overline{\Theta},v}_{\cdot+r}$ satisfies SDE
\begin{equation*}
	\left\{\begin{aligned}
		\d X_t=&\Big(A_{t+r}X_t+\overline{A}_{t+r}\mathbb{E}X_t+B_{t+r}(\Theta_{t+r}X_t+\overline{\Theta}_{t+r} \mathbb{E}X_t+v_{t+r})\\
		&+\overline{B}_{t+r}\mathbb{E}\big[\Theta_{t+r}X_t+\overline{\Theta}_{t+r}\mathbb{E}X_t+v_{t+r}\big]+b_{t+r}\Big)\d t\\
		&+\Big(C_{t+r}X_t+\overline{C}_{t+r}\mathbb{E}X_t+D_{t+r}(\Theta_{t+r}X_t+\overline{\Theta}_{t+r}\mathbb{E}X_t+v_{t+r})\\
		&+\overline{D}_{t+r}\mathbb{E}\big[\Theta_{t+r}X_t+\overline{\Theta}_{t+r}\mathbb{E}X_t+v_{t+r}\big]+\sigma_{t+r}\Big)\d W^{r}_t,\ \ \ t\ge 0,\\
		X_0=&x.
	\end{aligned}\right.
	\nonumber
\end{equation*}
On the other hand, by the periodicity of coefficients, $X^{k\tau+r,x,\Theta,\overline{\Theta},v}_{\cdot+k\tau+r}$ satisfies SDE
\begin{equation*}
	\left\{\begin{aligned}
		\d X_t=&\Big(A_{t+r}X_t+\overline{A}_{t+r}\mathbb{E}X_t+B_{t+r}(\Theta_{t+r}X_t+\overline{\Theta}_{t+r} \mathbb{E}X_t+v_{t+r})\\
		&+\overline{B}_{t+r}\mathbb{E}\big[\Theta_{t+r}X_t+\overline{\Theta}_{t+r}\mathbb{E}X_t+v_{t+r}\big]+b_{t+r}\Big)\d t\\
		&+\Big(C_{t+r}X_t+\overline{C}_{t+r}\mathbb{E}X_t+D_{t+r}(\Theta_{t+r}X_t+\overline{\Theta}_{t+r}\mathbb{E}X_t+v_{t+r})\\
		&+\overline{D}_{t+r}\mathbb{E}\big[\Theta_{t+r}X_t+\overline{\Theta}_{t+r}\mathbb{E}X_t+v_{t+r}\big]+\sigma_{t+r}\Big)\d W^{k\tau+r}_t,\ \ \ t\ge 0,\\
		X_0=&x.
	\end{aligned}\right.
\end{equation*}
By the uniqueness of weak solution to SDE, the periodicity of transition probability follows.
\end{proof}

Next we show that the long time behavior of the state equation can be depicted by periodic property.
\begin{proposition}\label{w_2 convergence theorem deterministic initial state} For a given $\tau>0$, assume $A,\overline{A},B,\overline{B},C,\overline{C},D,\overline{D}, b,\sigma\in\mathscr{B}_\tau$ and {\rm$\textbf{(A3)}$}, and $X^{r,x,\Theta,\overline{\Theta},v}$, $r\in[0,\tau)$ and $(\Theta,\overline{\Theta},v)\in\mathbb{U}$, is the solution of closed-loop system \eqref{wz2} with the initial time $r$ and initial state $x$. Then there exists a unique $\mu_r^{\Theta,\overline{\Theta},v}\in \mathcal{P}_2(\mathbb{R}^n)$, such that for any $x\in\mathbb{R}^n$,
\begin{equation}\label{wz13}
	\begin{aligned}		\lim\limits_{k\rightarrow\infty}w_2(p^{\Theta,\overline{\Theta},v}(r,r+k\tau,x;\cdot),\mu_r^{\Theta,\overline{\Theta},v}(\cdot))=0.
	\end{aligned}
\end{equation}
\end{proposition}
\begin{proof}
Denote
\begin{equation*}
	\begin{aligned}
		\Delta=\Big\{(f,g)\Big|f,g\in C_b(\mathbb{R}^n),\ f(y_1)-g(y_2)\le |y_1-y_2|^2\Big\}.
	\end{aligned}
	\nonumber
\end{equation*}
By Kantorovich duality (Villani \cite[Theorem 5.10]{ref5}), Propositions \ref{closed-loop estimate} and  \ref{periodic time shift property}, we have for any $r\in[0,\tau)$, $t_2=r+k_2\tau>t_1=r+k_1\tau\ge0$, $k_1,k_2\in\mathbb{N}$ and $x_1,x_2\in\mathbb{R}^n$,
\begin{equation*}
	\begin{aligned} &w_2(p^{\Theta,\overline{\Theta},v}(r,t_1,x_1;\cdot),p^{\Theta,\overline{\Theta},v}(r,t_2,x_2;\cdot))^2\\
		=&\sup\limits_{(f,g)\in\Delta} \left(\int_{\mathbb{R}^n}f(y_1)p^{\Theta,\overline{\Theta},v}(r,t_1,x_1;\d y_1)-\int_{\mathbb{R}^n}g(y_2)p^{\Theta,\overline{\Theta},v}(r,t_2,x_2;\d y_2)\right)\\
		=&\sup\limits_{(f,g)\in\Delta} \bigg(\int_{\mathbb{R}^n}f(y_1)p^{\Theta,\overline{\Theta},v}(r,t_1,x_1;\d y_1)\\
&\ \ \ \ \ \ \ \ \ \ \ -\int_{\mathbb{R}^n}g(y_2)\int_{\mathbb{R}^n}p^{\Theta,\overline{\Theta},v}(t_2-t_1+r,t_2,z;\d y_2)p^{\Theta,\overline{\Theta},v}(r,t_2-t_1+r,x_2;\d z)\bigg)\\ \le&\int_{\mathbb{R}^n}\sup\limits_{(f,g)\in\Delta}\left(\int_{\mathbb{R}^n}f(y_1)p^{\Theta,\overline{\Theta},v}(r,t_1,x_1;\d y_1)-\int_{\mathbb{R}^n}g(y_2)p^{\Theta,\overline{\Theta},v}(r,t_1,z;\d y_2)\right)\\
&\ \ \ \ \ \ \times p^{\Theta,\overline{\Theta},v}(r,t_2-t_1+r,x_2;\d z)\\
=&\int_{\mathbb{R}^n}w_2(p^{\Theta,\overline{\Theta},v}(r,t_1,x_1;\cdot),p^{\Theta,\overline{\Theta},v}(r,t_1,z;\cdot))^2p^{\Theta,\overline{\Theta},v}(r,t_2-t_1+r,x_2;\d z)\\		\le&\int_{\mathbb{R}^n}\mathbb{E}|X^{r,x_1,\Theta,\overline{\Theta},v}_{t_1}-X^{r,z,\Theta,\overline{\Theta},v}_{t_1}|^2p^{\Theta,\overline{\Theta},v}(r,t_2-t_1+r,x_2;\d z)\\
		\le&Ke^{-\frac{\lambda k_1\tau}{2}}\int_{\mathbb{R}^n}|x_1-z|^2p^{\Theta,\overline{\Theta},v}(r,t_2-t_1+r,x_2;\d z)\\
		\le&Ke^{-\frac{\lambda k_1\tau}{2}}\left(|x_1|^2+\mathbb{E}\Big[|X^{r,x_2,\Theta,\overline{\Theta},v}_{(k_2-k_1)\tau+r}|^2\Big]\right)\\
		\le&Ke^{-\frac{\lambda k_1\tau}{2}}\left(|x_1|^2+K(1+e^{-\frac{\lambda(k_2-k_1)\tau}{2}}|x_2|^2)\right).
	\end{aligned}
\end{equation*}
This yields that the transition probability sequence $\Big\{p^{\Theta,\overline{\Theta},v}(r,r+k\tau,x;\cdot)\Big\}_{k\in\mathbb{N}}$ is a Cauchy sequence
in $(\mathcal{P}_2(\mathbb{R}^n),w_2)$, so there exists a limit $\mu^{\Theta,\overline{\Theta},v}_r\in\mathcal{P}_2(\mathbb{R}^n)$ such that for any $x\in\mathbb{R}^n$, $\lim\limits_{k\to\infty}w_2(p^{\Theta,\overline{\Theta},v}(r,r+k\tau,x;\cdot),\mu^{\Theta,\overline{\Theta},v}_r(\cdot))=0$.

As for the uniqueness, we first prove that for any $r\in[0,\tau)$, $k\in\mathbb{N}$ and $B\in\mathscr{B}_{\mathbb{R}^n}$,
\begin{equation}\label{wz11}
	\begin{aligned}
\mu^{\Theta,\overline{\Theta},v}_r(B)=\int_{\mathbb{R}^n}p^{\Theta,\overline{\Theta},v}(r,r+k\tau,z;B)\mu^{\Theta,\overline{\Theta},v}_r(\d z).
	\end{aligned}
\end{equation}
To see this, we notice that by Proposition \ref{periodic time shift property}, for any $h\in C_b(\mathbb{R}^n)$, $r\in[0,\tau)$ and $k,k'\in\mathbb{N}$,
\begin{equation}\label{wz12}
	\begin{aligned}
		&\int_{\mathbb{R}^n}h(y)p^{\Theta,\overline{\Theta},v}(r,r+(k'+k)\tau,x;\d y)\\
=&\int_{\mathbb{R}^n}h(y)\int_{\mathbb{R}^n}p^{\Theta,\overline{\Theta},v}(r+k'\tau,r+(k'+k)\tau,z;\d y)p^{\Theta,\overline{\Theta},v}(r,r+k'\tau,x;\d z)\\
		=&\int_{\mathbb{R}^n}\int_{\mathbb{R}^n}h(y)p^{\Theta,\overline{\Theta},v}(r,r+k\tau,z;\d y)p^{\Theta,\overline{\Theta},v}(r,r+k'\tau,x;\d z).
	\end{aligned}
\end{equation}
Also by Lemma \ref{Feller property} $\int_{\mathbb{R}^n}h(y)p^{\Theta,\overline{\Theta},v}(r,r+k\tau,z;\d y)\in C_b(\mathbb{R}^n)$. Hence as $k'\rightarrow\infty$ in (\ref{wz12}) it leads to
\begin{equation*}
	\begin{aligned}
		\int_{\mathbb{R}^n}h(y)\mu_r(\d y)
		=\int_{\mathbb{R}^n}h(y)\int_{\mathbb{R}^n}p^{\Theta,\overline{\Theta},v}(r,r+k\tau,z;\d y)\mu_r(\d z),
	\end{aligned}
	\nonumber
\end{equation*}
which completes the proof of (\ref{wz11}). If there exists another $\mu_r^{\Theta,\overline{\Theta},v}\in \mathcal{P}_2(\mathbb{R}^n)$ independent of $x$ satisfying (\ref{wz13}). As $k\rightarrow\infty$,  by the dominated convergence theorem we have
\begin{equation*}
	\begin{aligned}
		\int_{\mathbb{R}^n}h(y)\mu_r'(\d y)&
=\int_{\mathbb{R}^n}\int_{\mathbb{R}^n}h(y)p^{\Theta,\overline{\Theta},v}(r,r+k\tau,z;\d y)\mu_r'(\d z)\\
		&\rightarrow\int_{\mathbb{R}^n}\int_{\mathbb{R}^n}h(y)\mu_r(\d y)\mu_r'(\d z)=\int_{\mathbb{R}^n}h(y)\mu_r(\d y),
	\end{aligned}
	\nonumber
\end{equation*}
which implies $\mu_r=\mu_r'$.
\end{proof}

Moreover, the convergence in Proposition \ref{w_2 convergence theorem deterministic initial state} still holds with a single initial state replaced by a distributed initial state.
\begin{proposition}\label{w_2 convergence theorem random initial state}
Assume the same conditions as in Proposition \ref{w_2 convergence theorem deterministic initial state}. Then there exists a unique $\mu_r^{\Theta,\overline{\Theta},v}\in \mathcal{P}_2(\mathbb{R}^n)$, such that for any $\xi_r\in \mathscr{F}$ satisfying that $\mathbb{E}|\xi_r|^2<\infty$ and $\xi_r$ is independent of $\sigma\{W_t-W_r,\ t\geq r\}$,
\begin{equation*}
	\begin{aligned} \lim\limits_{k\rightarrow\infty}w_2(p^{\Theta,\overline{\Theta},v}(r,r+k\tau,\xi_r;\cdot),\mu_r^{\Theta,\overline{\Theta},v}(\cdot))=0.
	\end{aligned}
\end{equation*}
\end{proposition}
\begin{proof}
Note that by Proposition \ref{closed-loop estimate}, for any $r\in[0,\tau)$ and $k\in\mathbb{N}$,
\begin{equation}
	\begin{aligned}
\int_{\mathbb{R}^n}|y|^2 p^{\Theta,\overline{\Theta},v}(r,r+k\tau,x;\d y)\le K(1+|x|^2),
	\end{aligned}
	\nonumber
\end{equation}
and for any $\xi_r\in \mathscr{F}$ satisfying that $\mathbb{E}|\xi_r|^2<\infty$ and $\xi_r$ is independent of $\sigma\{W_t-W_r,\ t\geq r\}$,
\begin{equation}
	\begin{aligned}
\int_{\mathbb{R}^n}K(1+|x|^2)\mathbb{P}\big(\xi_r\in\d x\big)\le K(1+\mathbb{E}|\xi_r|^2)<\infty.
	\end{aligned}
	\nonumber
\end{equation}
Due to (\ref{wz13}), by the dominated convergence theorem and Lemma \ref{w_2 convergence equivalent condition}, we have for any $h\in C_b(\mathbb{R}^n)$,
\begin{equation*}
	\begin{aligned}
	\int_{\mathbb{R}^n}h(y)\mathbb{P}(X^{r,\xi_r,\Theta,\overline{\Theta},v}_{r+k\tau}\in\d y)&=\int_{\mathbb{R}^n}\int_{\mathbb{R}^n}h(y)p^{\Theta,\overline{\Theta},v}(r,r+k\tau,x;\d y)\mathbb{P}\big(\xi_r\in\d x\big)\\
	&\rightarrow\int_{\mathbb{R}^n}\int_{\mathbb{R}^n}h(y)\mu_r(\d y)\mathbb{P}\big(\xi_r\in\d x\big)=\int_{\mathbb{R}^n}h(y)\mu_r(\d y)
		\end{aligned}
		\nonumber
\end{equation*}
and
\begin{equation*}
	\begin{aligned}
\int_{\mathbb{R}^n}|y|^2\mathbb{P}(X^{r,\xi_r,\Theta,\overline{\Theta},v}_{r+k\tau}\in\d y)&=\int_{\mathbb{R}^n}\int_{\mathbb{R}^n}|y|^2p^{\Theta,\overline{\Theta},v}(r,r+k\tau,x;\d y)\mathbb{P}\big(\xi_r\in\d x\big)\\
	&\rightarrow\int_{\mathbb{R}^n}\int_{\mathbb{R}^n}|y|^2\mu_r(\d y)\mathbb{P}\big(\xi_r\in\d x\big)=\int_{\mathbb{R}^n}|y|^2\mu_r(\d y).
		\end{aligned}
		\nonumber
\end{equation*}
By Lemma \ref{w_2 convergence equivalent condition} again, the desired conclusion follows.
\end{proof}

The next proposition connects all measures $\mu_r^{\Theta,\overline{\Theta},v}$, $r\in[0,\tau)$, by the stochastic flow.
\begin{proposition}\label{periodic measure stochastic flow property}
For a given $\tau>0$, assume $A,\overline{A},B,\overline{B},C,\overline{C},D,\overline{D}, b,\sigma\in\mathscr{B}_\tau$ and {\rm$\textbf{(A3)}$}, $\xi\in L_{\mathscr{F}}^2(\Omega;\mathbb{R}^n)$ is independent of $\mathbb{F}$ satisfying $\mathbb{P}_{\xi}=\mu_0^{\Theta,\overline{\Theta},v}$ for $(\Theta,\overline{\Theta},v)\in\mathbb{U}$, and $X^{0,\xi,\Theta,\overline{\Theta},v}$ is the solution of closed-loop system \eqref{wz2} with the initial time $0$ and initial state $\xi$. Then
\begin{equation*}
	\begin{aligned}
		\mathbb{P}_{X^{0,\xi,\Theta,\overline{\Theta},v}_r}=\mu_r^{\Theta,\overline{\Theta},v},\ \ \ r\in[0,\tau).
	\end{aligned}
\end{equation*}
\end{proposition}
\begin{proof}
By (\ref{wz11}),
\begin{equation*}
	\begin{aligned}
		\mu_0^{\Theta,\overline{\Theta},v}(\d x)=\int_{\mathbb{R}^n}p^{\Theta,\overline{\Theta},v}(0,\tau,y;\d x)\mu_0^{\Theta,\overline{\Theta},v}(\d y).
	\end{aligned}
	\nonumber
\end{equation*}
which is equivalent to
$\mathbb{P}(\xi\in\d x)=\mathbb{P}(X^{0,\xi,\Theta,\overline{\Theta},v}_\tau\in\d x)$.
This, together with Proposition \ref{periodic time shift property}, leads to for any $t\ge 0$ and $B\in\mathscr{B}_{\mathbb{R}^n}$,
\begin{equation*}
	\begin{aligned}
		\mathbb{P}(X^{0,\xi,\Theta,\overline{\Theta},v}_t\in B)&=\int_{\mathbb{R}^n}\mathbb{P}(X^{0,x,\Theta,\overline{\Theta},v}_t\in B)\mathbb{P}(\xi\in\d x)\\
&=\int_{\mathbb{R}^n}\mathbb{P}(X^{\tau,x,\Theta,\overline{\Theta},v}_{t+\tau}\in B)\mathbb{P}(X^{0,\xi,\Theta,\overline{\Theta},v}_\tau\in\d x)\\
&=\mathbb{P}(X^{0,\xi,\Theta,\overline{\Theta},v}_{t+\tau}\in B).
	\end{aligned}
\end{equation*}
In particular, for any $k\in\mathbb{N},r\in[0,\tau)$, by the property of stochastic flow,
\begin{equation}\label{wz14}
	\begin{aligned}
		\mathbb{P}(X^{0,\xi,\Theta,\overline{\Theta},v}_r\in B)=&\mathbb{P}(X^{0,\xi,\Theta,\overline{\Theta},v}_{r+k\tau}\in B)=\mathbb{P}(X^{r,X^{0,\xi,\Theta,\overline{\Theta},v}_r,\Theta,\overline{\Theta},v}_{r+k\tau}\in B).
	\end{aligned}
\end{equation}
On the other hand, by Proposition \ref{closed-loop estimate}, for any $r\in[0,\tau)$,
$\mathbb{E}\Big[|X^{0,\xi,\Theta,\overline{\Theta},v}_r|^2\Big]\le K(1+\mathbb{E}|\xi|^2)<\infty$, so we use Proposition \ref{w_2 convergence theorem random initial state} to have
\begin{equation*}
	\begin{aligned} \lim\limits_{k\rightarrow\infty}w_2(\mathbb{P}_{X^{r,X^{0,\xi,\Theta,\overline{\Theta},v}_r,\Theta,\overline{\Theta},v}_{r+k\tau}},\mu^{\Theta,\overline{\Theta},v}_r)=0.
	\end{aligned}
\end{equation*}
Thus as $k\to\infty$ in (\ref{wz14}) Proposition \ref{periodic measure stochastic flow property} follows immediately.
\end{proof}

Actually, the measures $\mu_r^{\Theta,\overline{\Theta},v}\in \mathcal{P}_2(\mathbb{R}^n)$, $r\in[0,\tau)$ and $(\Theta,\overline{\Theta},v)\in\mathbb{U}$, in Proposition \ref{w_2 convergence theorem deterministic initial state} can be extended to $[0,\infty)$, and in this way we get a periodic measure as defined in \cite{fe-zh}.
\begin{definition}\label{wz23}
For $r\in[0,\tau)$, $k\in\mathbb{N}$, $t=k\tau+r$ and $(\Theta,\overline{\Theta},v)\in\mathbb{U}$, define
$$\mu^{\Theta,\overline{\Theta},v}_t=\mu^{\Theta,\overline{\Theta},v}_r,$$
where $\mu_r^{\Theta,\overline{\Theta},v}\in \mathcal{P}_2(\mathbb{R}^n)$, $r\in[0,\tau)$ is given in Proposition \ref{w_2 convergence theorem deterministic initial state}.
\end{definition}
\begin{proposition}\label{wz16}
For a given $\tau>0$, assume $A,\overline{A},B,\overline{B},C,\overline{C},D,\overline{D}, b,\sigma\in\mathscr{B}_\tau$ and {\rm$\textbf{(A3)}$}, and
$\mu^{\Theta,\overline{\Theta},v}_t$, $t\in[0,\infty)$ and $(\Theta,\overline{\Theta},v)\in\mathbb{U}$, is a family of measures in $\mathcal{P}_2(\mathbb{R}^n)$ defined in Definition \ref{wz23}. Then for $s,t\ge0$, $x\in\mathbb{R}^n$ and $B\in\mathscr{B}_{\mathbb{R}^n}$,
	\begin{equation}\label{wz15}
		\begin{aligned} \mu^{\Theta,\overline{\Theta},v}_{t+s}(B)=\int_{\mathbb{R}^n}p^{\Theta,\overline{\Theta},v}(s,t+s,x;B)\mu^{\Theta,\overline{\Theta},v}_s(\d x).
		\end{aligned}
	\end{equation}
\end{proposition}
\begin{proof}
In the first case that $s\in[0,\tau)$, let $t+s=k'\tau+r'$, where $k'\in\mathbb{N}$ and $r'\in[0,\tau)$.
For $B\in\mathscr{B}_{\mathbb{R}^n}$, by Propositions \ref{periodic measure stochastic flow property} and \ref{periodic time shift property} we have
\begin{equation*}
	\begin{aligned}
		&\int_{\mathbb{R}^n}p^{\Theta,\overline{\Theta},v}(s,t+s,x;B)\mu^{\Theta,\overline{\Theta},v}_s(\d x)\\
=&\int_{\mathbb{R}^n}p^{\Theta,\overline{\Theta},v}(s,t+s,x;B)\int_{\mathbb{R}^n}p^{\Theta,\overline{\Theta},v}(0,s,z;\d x)\mu^{\Theta,\overline{\Theta},v}_0(\d z)\\
		=&\int_{\mathbb{R}^n}p^{\Theta,\overline{\Theta},v}(0,t+s,z;B)\mu^{\Theta,\overline{\Theta},v}_0(\d z)\\
=&\int_{\mathbb{R}^n}p^{\Theta,\overline{\Theta},v}(k'\tau,k'\tau+r',x;B)\int_{\mathbb{R}^n}p^{\Theta,\overline{\Theta},v}(0,k'\tau,z;\d x)\mu^{\Theta,\overline{\Theta},v}_0(\d z)\\		=&\int_{\mathbb{R}^n}p^{\Theta,\overline{\Theta},v}(k'\tau,k'\tau+r',x;B)\mu^{\Theta,\overline{\Theta},v}_0(\d x)\\		=&\int_{\mathbb{R}^n}p^{\Theta,\overline{\Theta},v}(0,r',x;B)\mu^{\Theta,\overline{\Theta},v}_0(\d x)		=\mu^{\Theta,\overline{\Theta},v}_{r'}(B)
		=\mu^{\Theta,\overline{\Theta},v}_{k'\tau+r'}(B)
		=\mu^{\Theta,\overline{\Theta},v}_{t+s}(B).
	\end{aligned}
	\nonumber
\end{equation*}
Hence (\ref{wz15}) holds for $s\in[0,\tau)$. In the general case that $s=k''\tau+r''$, where $k''\in\mathbb{N}$ and $r''\in[0,\tau)$, for $B\in\mathscr{B}_{\mathbb{R}^n}$, based on the first case we have
\begin{equation*}
	\begin{aligned}
\int_{\mathbb{R}^n}p^{\Theta,\overline{\Theta},v}(s,t+s,x;B)\mu^{\Theta,\overline{\Theta},v}_s(\d x)
=&\int_{\mathbb{R}^n}p^{\Theta,\overline{\Theta},v}(r'',t+r'',x;B)\mu^{\Theta,\overline{\Theta},v}_{r''}(\d x)\\
		=&\mu^{\Theta,\overline{\Theta},v}_{t+r''}(B)
		=\mu^{\Theta,\overline{\Theta},v}_{t+k''\tau+r''}(B)
		=\mu^{\Theta,\overline{\Theta},v}_{t+s}(B).
	\end{aligned}
\end{equation*}
The proof of Proposition \ref{wz16} is complete.
\end{proof}

\section{Ergodic optimal control problem}

We back to \textbf{Problem (CL-MFLQE)}. First we rewrite the cost functional $\mathcal{E}(x,u)$ to an equivalent form on a finite horizon.
\begin{theorem}\label{transffered form by periodic measure} Assume {\rm$\textbf{(A1)}$} and {\rm$\textbf{(A3)}$}, $X^{0,x,\Theta,\overline{\Theta},v}$ and $X^{0,\xi,\Theta,\overline{\Theta},v}$, $(\Theta,\overline{\Theta},v)\in\mathbb{U}$, are the solutions of closed-loop system \eqref{wz2} with initial state $x\in\mathbb{R}^n$ and $\xi\in L_{\mathscr{F}}^2(\Omega;\mathbb{R}^n)$, respectively, where $\xi$ is independent of $\mathbb{F}$ satisfying $\mathbb{P}_{\xi}=\mu_0^{\Theta,\overline{\Theta},v}$, and $u^{0,x,\Theta,\overline{\Theta},v}_t=\Theta_tX^{0,x,\Theta,\overline{\Theta},v}_t
+\overline{\Theta}_t\mathbb{E}X^{0,x,\Theta,\overline{\Theta},v}_t+v_t$ and $u^{0,\xi,\Theta,\overline{\Theta},v}=\Theta_tX^{0,\xi,\Theta,\overline{\Theta},v}_t
+\overline{\Theta}_t\mathbb{E}X^{0,\xi,\Theta,\overline{\Theta},v}_t+v_t$. Then
\begin{equation*}
	\begin{aligned}
		&\varliminf\limits_{T \to \infty}\frac{1}{T}\mathbb{E}\int_0^TF(t,X^{0,x,\Theta,\overline{\Theta},v}_t,\mathbb{E}X^{0,x,\Theta,\overline{\Theta},v}_t,u^{0,x,\Theta,\overline{\Theta},v}_t,\mathbb{E}u^{0,x,\Theta,\overline{\Theta},v}_t)\d t\\
		=&\frac{1}{\tau}\mathbb{E}\int_0^\tau F(r,X^{0,\xi,\Theta,\overline{\Theta},v}_t,\mathbb{E}X^{0,\xi,\Theta,\overline{\Theta},v}_t,u^{0,\xi,\Theta,\overline{\Theta},v}_t,\mathbb{E}u^{0,\xi,\Theta,\overline{\Theta},v}_t)\d t.
	\end{aligned}
\end{equation*}
\end{theorem}
\begin{proof}
For any $T>0$, there exists $N\in\mathbb{N}$ and $l\in[0,\tau)$ such that $T=N\tau+l$. Hence
{\small\begin{equation}\label{wz17}
	\begin{aligned}		&\frac{1}{T}\mathbb{E}\int_0^TF(t,X^{0,x,\Theta,\overline{\Theta},v}_t,\mathbb{E}X^{0,x,\Theta,\overline{\Theta},v}_t,u^{0,x,\Theta,\overline{\Theta},v}_t,\mathbb{E}u^{0,x,\Theta,\overline{\Theta},v}_t)\d t\\			
=&\left(\frac{1}{T}-\frac{1}{N\tau}\right)\mathbb{E}\int_0^TF(t,X^{0,x,\Theta,\overline{\Theta},v}_t,\mathbb{E}X^{0,x,\Theta,\overline{\Theta},v}_t,u^{0,x,\Theta,\overline{\Theta},v}_t,\mathbb{E}u^{0,x,\Theta,\overline{\Theta},v}_t)\d t\\		&+\frac{1}{N\tau}\mathbb{E}\int_{N\tau}^TF(t,X^{0,x,\Theta,\overline{\Theta},v}_t,\mathbb{E}X^{0,x,\Theta,\overline{\Theta},v}_t,u^{0,x,\Theta,\overline{\Theta},v}_t,\mathbb{E}u^{0,x,\Theta,\overline{\Theta},v}_t)\d t\\		&+\frac{1}{N\tau}\mathbb{E}\sum_{k=0}^{N-1}\int_{k\tau}^{(k+1)\tau}F(t,X^{0,x,\Theta,\overline{\Theta},v}_t,\mathbb{E}X^{0,x,\Theta,\overline{\Theta},v}_t,u^{0,x,\Theta,\overline{\Theta},v}_t,\mathbb{E}u^{0,x,\Theta,\overline{\Theta},v}_t)\d t\\
		=&-\frac{l}{N\tau T}\mathbb{E}\int_0^TF(t,X^{0,x,\Theta,\overline{\Theta},v}_t,\mathbb{E}X^{0,x,\Theta,\overline{\Theta},v}_t,u^{0,x,\Theta,\overline{\Theta},v}_t,\mathbb{E}u^{0,x,\Theta,\overline{\Theta},v}_t)\d t\\		&+\frac{1}{N\tau}\mathbb{E}\int_0^lF(t,X^{0,x,\Theta,\overline{\Theta},v}_{t+N\tau},\mathbb{E}X^{0,x,\Theta,\overline{\Theta},v}_{t+N\tau},u^{0,x,\Theta,\overline{\Theta},v}_{t+N\tau},\mathbb{E}u^{0,x,\Theta,\overline{\Theta},v}_{t+N\tau})\d t\\	&+\frac{1}{N\tau}\mathbb{E}\sum_{k=0}^{N-1}\int_{0}^{\tau}F(t,X^{0,x,\Theta,\overline{\Theta},v}_{t+k\tau},\mathbb{E}X^{0,x,\Theta,\overline{\Theta},v}_{t+k\tau},u^{0,x,\Theta,\overline{\Theta},v}_{t+k\tau},\mathbb{E}u^{0,x,\Theta,\overline{\Theta},v}_{t+k\tau})\d t.
	\end{aligned}
\end{equation}}
Since
{\small\begin{equation*}
	\begin{aligned} &F(t,X^{0,x,\Theta,\overline{\Theta},v}_t,\mathbb{E}X^{0,x,\Theta,\overline{\Theta},v}_t,u^{0,x,\Theta,\overline{\Theta},v}_t,\mathbb{E}u^{0,x,\Theta,\overline{\Theta},v}_t)\\
		=&\left\langle(Q_t+S_t^\top\Theta_t+\Theta_t^\top S_t+\Theta_t^\top R_t\Theta_t)X^{0,x,\Theta,\overline{\Theta},v}_t,X^{0,x,\Theta,\overline{\Theta},v}_t\right\rangle+\left\langle R_tv_t,v_t\right\rangle\\
		&+2\left\langle(S_t+R_t\Theta_t)X^{0,x,\Theta,\overline{\Theta},v}_t,v_t\right\rangle+2\left\langle X^{0,x,\Theta,\overline{\Theta},v}_t,q_t+\Theta_t^\top\rho_t\right\rangle+2\left\langle\rho_t,v_t\right\rangle\\
		&+\Big\langle(S_t^\top\overline{\Theta}_t+\overline{\Theta}_t^\top S_t+\Theta_t^\top R_t\overline{\Theta}_t+\overline{\Theta}_t^\top R_t\Theta_t)X^{0,x,\Theta,\overline{\Theta},v}_t,\mathbb{E}X^{0,x,\Theta,\overline{\Theta},v}_t\Big\rangle\\
		&+\Big\langle(\overline{Q}_t+\overline{\Theta}_t^\top R_t\overline{\Theta}_t+\overline{S}_t^\top\widehat{\Theta}_t+\widehat{\Theta}_t^\top\overline{S}_t+\widehat{\Theta}_t^\top\overline{R}_t\widehat{\Theta}_t)\mathbb{E}X^{0,x,\Theta,\overline{\Theta},v}_t,\mathbb{E}X^{0,x,\Theta,\overline{\Theta},v}_t\Big\rangle\\ &+2\Big\langle(\overline{S}_t+R_t\overline{\Theta}_t+\overline{R}_t\widehat{\Theta}_t)\mathbb{E}X^{0,x,\Theta,\overline{\Theta},v}_t,v_t\Big\rangle+2\left\langle X^{0,x,\Theta,\overline{\Theta},v}_t,\overline{\Theta}_t^\top\rho_t\right\rangle+\left\langle \overline{R}_tv_t,v_t\right\rangle,
	\end{aligned}
\end{equation*}}
there exists $K^a,K^b\in\mathscr{B}_\tau(\mathbb{S}^n)$, $K^c\in\mathscr{B}_\tau(\mathbb{R}^n)$ and $K^d\in\mathscr{B}_\tau(\mathbb{R})$ such that for any $t\ge 0$,
{\small\begin{equation*}
	\begin{aligned}	&\mathbb{E}F(t,X^{0,x,\Theta,\overline{\Theta},v}_t,\mathbb{E}X^{0,x,\Theta,\overline{\Theta},v}_t,u^{0,x,\Theta,\overline{\Theta},v}_t,\mathbb{E}u^{0,x,\Theta,\overline{\Theta},v}_t)\\
	=&\mathbb{E}\big\langle K^a_tX^{0,x,\Theta,\overline{\Theta},v}_t,X^{0,x,\Theta,\overline{\Theta},v}_t\big\rangle+\big\langle K^b_t\mathbb{E}X^{0,x,\Theta,\overline{\Theta},v}_t,\mathbb{E}X^{0,x,\Theta,\overline{\Theta},v}_t\big\rangle+\big\langle K^c_t,\mathbb{E}X^{0,x,\Theta,\overline{\Theta},v}_t\big\rangle+K^d_t.
	\end{aligned}
\end{equation*}}
Then for the first and second terms on the right hand of (\ref{wz17}), as $N\rightarrow\infty$ which is equivalent to $N\rightarrow\infty$, we have
{\small\begin{equation*}
	\begin{split}
		&\Big|\frac{l}{N\tau T}\mathbb{E}\int_0^TF(t,X^{0,x,\Theta,\overline{\Theta},v}_t,\mathbb{E}X^{0,x,\Theta,\overline{\Theta},v}_t,u^{0,x,\Theta,\overline{\Theta},v}_t,\mathbb{E}u^{0,x,\Theta,\overline{\Theta},v}_t)\d t\Big|\\		&+\Big|\frac{1}{N\tau}\mathbb{E}\int_0^lF(t,X^{0,x,\Theta,\overline{\Theta},v}_{t+N\tau},\mathbb{E}[X^{0,x,\Theta,\overline{\Theta},v}_{t+N\tau}],u^{0,x,\Theta,\overline{\Theta},v}_{t+N\tau},\mathbb{E}[u^{0,x,\Theta,\overline{\Theta},v}_{t+N\tau}])\d t\Big|\\
		\le&\frac{l}{N\tau}\sup\limits_{t\ge 0}K\Big(1+|\mathbb{E}X^{0,x,\Theta,\overline{\Theta},v}_t|+\mathbb{E}|X^{0,x,\Theta,\overline{\Theta},v}_t|^2+(\mathbb{E}|X^{0,x,\Theta,\overline{\Theta},v}_t|)^2\Big)\\
		\le&\frac{K}{N}(1+|x|^2)\to0.
	\end{split}
\end{equation*}}
To deal with the last term on the right hand of (\ref{wz17}), first note that for any $r\in[0,\tau)$, $X^{0,x,\Theta,\overline{\Theta},v}_{r+k\tau}=X^{r,X^{0,x,\Theta,\overline{\Theta},v}_r,\Theta,\overline{\Theta},v}_{r+k\tau}$.
By Proposition \ref{w_2 convergence theorem random initial state} it yields that
{\small\begin{equation*}
	\begin{aligned}		&\lim\limits_{k\rightarrow\infty}w_2(p^{\Theta,\overline{\Theta},v}(0,r+k\tau,x;\cdot),\mu_r^{\Theta,\overline{\Theta},v}(\cdot))\\
=&\lim\limits_{k\rightarrow\infty}w_2(p^{\Theta,\overline{\Theta},v}(r,r+k\tau,X^{0,x,\Theta,\overline{\Theta},v}_r;\cdot),\mu_r^{\Theta,\overline{\Theta},v}(\cdot))=0.
	\end{aligned}
\end{equation*}}
Hence by Lemma \ref{w_2 convergence equivalent condition}, together with Stolz theorem, as $N\rightarrow\infty$ we have
{\footnotesize\begin{equation*}
	\begin{aligned} &\mathbb{E}\frac{1}{N}\sum_{k=0}^{N-1}F(r,X^{0,x,\Theta,\overline{\Theta},v}_{r+k\tau},\mathbb{E}X^{0,x,\Theta,\overline{\Theta},v}_{r+k\tau},u^{0,x,\Theta,\overline{\Theta},v}_{r+k\tau},\mathbb{E}u^{0,x,\Theta,\overline{\Theta},v}_{r+k\tau})\\
		=&\int_{\mathbb{R}^n}\Big\langle K^a_ry,y\Big\rangle\frac{1}{N}\sum_{k=0}^{N-1}p^{\Theta,\overline{\Theta},v}(0,r+k\tau,x;\d y)\\
&+\frac{1}{N}\sum_{k=0}^{N-1}\Big\langle K^b_r\int_{\mathbb{R}^n}yp^{\Theta,\overline{\Theta},v}(0,r+k\tau,x;\d y),
\int_{\mathbb{R}^n}yp^{\Theta,\overline{\Theta},v}(0,r+k\tau,x;\d y)\Big\rangle\\
&+\Big\langle K^c_r,\int_{\mathbb{R}^n}y\frac{1}{N}\sum_{k=0}^{N-1}p^{\Theta,\overline{\Theta},v}(0,r+k\tau,x;\d y)\Big\rangle+K^d_r\\
		\rightarrow&\int_{\mathbb{R}^n}\Big\langle K^a_ry,y\Big\rangle\mu_r^{\Theta,\overline{\Theta},v}(\d y)+\Big\langle K^b_r\int_{\mathbb{R}^n}y\mu_r^{\Theta,\overline{\Theta},v}(\d y),\int_{\mathbb{R}^n}y\mu_r^{\Theta,\overline{\Theta},v}(\d y)\Big\rangle\\
		&+\Big\langle K^c_r,\int_{\mathbb{R}^n}y\mu_r^{\Theta,\overline{\Theta},v}(\d y)\Big\rangle+K^d_r\\		=&F(r,X^{0,\xi,\Theta,\overline{\Theta},v}_r,\mathbb{E}X^{0,\xi,\Theta,\overline{\Theta},v}_r,u^{0,\xi,\Theta,\overline{\Theta},v}_r,\mathbb{E}u^{0,\xi,\Theta,\overline{\Theta},v}_r).
	\end{aligned}
	\nonumber
\end{equation*}}
Then it follows from the dominated convergence theorem that, as $N\rightarrow\infty$,
{\footnotesize\begin{equation*}
	\begin{split}		&\frac{1}{N\tau}\mathbb{E}\sum_{k=0}^{N-1}\int_{0}^{\tau}F(r,X^{0,x,\Theta,\overline{\Theta},v}_{r+k\tau},\mathbb{E}X^{0,x,\Theta,\overline{\Theta},v}_{r+k\tau},u^{0,x,\Theta,\overline{\Theta},v}_{r+k\tau},\mathbb{E}u_{r+k\tau})\d r\\
		\rightarrow&\frac{1}{\tau}\mathbb{E}\int_0^\tau F(r,X^{0,\xi,\Theta,\overline{\Theta},v}_r,\mathbb{E}X^{0,\xi,\Theta,\overline{\Theta},v}_r,u^{0,\xi,\Theta,\overline{\Theta},v}_r,\mathbb{E}[u^{0,\xi,\Theta,\overline{\Theta},v}_r])\d r.
	\end{split}
\end{equation*}}
We put an end of the proof for Proposition \ref{transffered form by periodic measure} by taking $T\rightarrow\infty$ in (\ref{wz17}).
\end{proof}

The equivalent form of cost functional provides a straightforward way to investigate the ergodic control problem on one periodic interval.
\begin{proposition}\label{transffered form by ODE} Assume the same conditions as in Theorem \ref{transffered form by periodic measure}. Then for any $\Pi,P\in\mathscr{D}_\tau(\mathbb{S}^{n})$ and $\eta\in\mathscr{D}_\tau(\mathbb{R}^{n})$,
{\footnotesize\begin{equation*}
	\begin{split}
	&\mathcal{E}(x,u^{0,x,\Theta,\overline{\Theta},v})\\
	=&\frac{1}{\tau}\int_{0}^{\tau}\Big(\mathbb{E}\Big\langle (\dot{P}_t+Q^{\Theta,P}_t)(X^{0,\xi,\Theta,\overline{\Theta},v}_t-\mathbb{E}X^{0,\xi,\Theta,\overline{\Theta},v}_t),(X^{0,\xi,\Theta,\overline{\Theta},v}_t-\mathbb{E}X^{0,\xi,\Theta,\overline{\Theta},v}_t)\Big\rangle\\
	&\ \ +\Big\langle (\dot{\Pi}_t+Q^{\widehat{\Theta},\Pi,P}_t)\mathbb{E}X^{0,\xi,\Theta,\overline{\Theta},v}_t,\mathbb{E}X^{0,\xi,\Theta,\overline{\Theta},v}_t\Big\rangle+2\Big\langle\mathbb{E}X^{0,\xi,\Theta,\overline{\Theta},v}_t,K^{\widehat{\Theta},\Pi,P}_tv_t\Big\rangle\\ &\ \  +2\left\langle\mathbb{E}X^{0,\xi,\Theta,\overline{\Theta},v}_t,\dot{\eta}_t+L^{\widehat{\Theta},\Pi,P,\eta}_t\right\rangle+\left\langle \widehat{R}^P_tv_t,v_t\right\rangle+2\left\langle v_t,l^{P,\eta}_t\right\rangle+\left\langle P_t\sigma_t,\sigma_t\right\rangle+2\left\langle\eta_t,b_t\right\rangle\Big)\d t,
	\end{split}
\end{equation*}}
where
{\footnotesize\begin{equation*}\label{cost functional coefficients}
	\begin{split}
		Q^{\Theta,P}_t=&(A_t+B_t\Theta_t)^\top P_t+P_t(A_t+B_t\Theta_t)+Q_t+\Theta_t^\top R_t\Theta_t\\
		&+(C_t+D_t\Theta_t)^\top P_t(C_t+D_t\Theta_t)+S_t^\top\Theta_t+\Theta_t^\top S_t,\\
		Q^{\widehat{\Theta},\Pi,P}_t=&(\widehat{A}_t+\widehat{B}_t\widehat{\Theta}_t)^\top \Pi_t+\Pi_t(\widehat{A}_t+\widehat{B}_t\widehat{\Theta}_t)+\widehat{Q}_t+\widehat{\Theta}_t^\top\widehat{R}_t\widehat{\Theta}_t\\
		&+(\widehat{C}_t+\widehat{D}_t\widehat{\Theta}_t)^\top P_t(\widehat{C}_t+\widehat{D}_t\widehat{\Theta}_t)+\widehat{S}_t^\top\widehat{\Theta}_t+\widehat{\Theta}_t^\top\widehat{S}_t,\\
		K^{\widehat{\Theta},\Pi,P}_t=&(\widehat{C}_t+\widehat{D}_t\widehat{\Theta}_t)^\top P_t\widehat{D}_t+(\widehat{S}_t+\widehat{R}_t\widehat{\Theta}_t)^\top+\Pi_t\widehat{B}_t,\\ L^{\widehat{\Theta},\Pi,P,\eta}_t=&(\widehat{A}_t+\widehat{B}_t\widehat{\Theta}_t)^\top\eta_t+(\widehat{C}_t+\widehat{D}_t\widehat{\Theta}_t)^\top P_t\sigma_t+\Pi_tb_t+q_t+\widehat{\Theta}_t^\top\rho_t,\\
		\widehat{R}^P_t=&\widehat{R}_t+\widehat{D}_t^\top P_t\widehat{D}_t,\ \ \ \ \ l^{P,\eta}_t=\widehat{B}_t^\top\eta_t+\widehat{D}_t^\top P_t\sigma_t+\rho_t.
	\end{split}
\end{equation*}}
\end{proposition}
\begin{proof}
To save space, we write $X^{0,\xi,\Theta,\overline{\Theta},v}_\cdot$ as $X_\cdot$ in this proof.
For $t\geq0$, denote
$Y_t=\mathbb{E}X_t$ and $Z_t= X_t-Y_t$.
For any $(\Pi,P,\eta)\in\mathscr{D}_\tau(\mathbb{S}^{n})\times\mathscr{D}_\tau(\mathbb{S}^{n})\times\mathscr{D}_\tau(\mathbb{R}^{n})$, applying It\^{o} formula to $\left\langle\Pi_\cdot Y_\cdot,Y_\cdot\right\rangle+\left\langle P_\cdot Z_\cdot,Z_\cdot\right\rangle+2\left\langle\eta_\cdot,Y_\cdot\right\rangle$, by Theorem \ref{transffered form by periodic measure} we have
{\footnotesize\begin{equation}\label{wz18}
	\begin{aligned}
		0
		=&\mathbb{E}\int_{0}^{\tau}\d\Big(\left\langle\Pi_tY_t,Y_t\right\rangle+\left\langle P_t Z_t,Z_t\right\rangle+2\left\langle\eta_t,Y_t\right\rangle\Big)\\
		=&\mathbb{E}\int_{0}^{\tau}\Big(\left\langle\dot{\Pi}_tY_t,Y_t\right\rangle+\left\langle\dot{P}_t Z_t,Z_t\right\rangle+2\left\langle\dot{\eta}_t,Y_t\right\rangle+\left\langle\Pi_t Y_t,(\widehat{A}_t+\widehat{B}_t\widehat{\Theta}_t)Y_t+\widehat{B}_tv_t+b_t\right\rangle\\		
&\ \ \ \ \ \ +\left\langle\Pi_t\left((\widehat{A}_t+\widehat{B}_t\widehat{\Theta}_t)Y_t+\widehat{B}_tv_t+b_t\right),Y_t\right\rangle+\left\langle P_tZ_t,(A_t+B_t\Theta_t)Z_t\right\rangle\\
		&\ \ \ \ \ \ +\left\langle P_t(A_t+B_t\Theta_t)Z_t,Z_t\right\rangle+\Big\langle P_t\Big((C_t+D_t\Theta_t)Z_t+(\widehat{C}_t+\widehat{D}_t\widehat{\Theta}_t)Y_t+\widehat{D}_tv_t+\sigma_t\Big),\\		&\ \ \ \ \ \ \ \ \ \ \ \ \ \ \ \ \ \ \ \ \ \ \ \ \ \ \ \ \ \ \ \ \ \ \ \ \ \ \ \ \ \ \ \ (C_t+D_t\Theta_t)Z_t+(\widehat{C}_t+\widehat{D}_t\widehat{\Theta}_t)Y_t+\widehat{D}_tv_t+\sigma_t\Big\rangle\\		&\ \ \ \ \ \ +2\left\langle\eta_t,(\widehat{A}_t+\widehat{B}_t\widehat{\Theta}_t)Y_t+\widehat{B}_tv_t+b_t\right\rangle\Big)\d t\\
		=&\int_{0}^{\tau}\Big(\mathbb{E}\Big\langle\Big(\dot{P}_t+(A_t+B_t\Theta_t)^\top P_t+P_t(A_t+B_t\Theta_t)+(C_t+D_t\Theta_t)^\top P_t(C_t+D_t\Theta_t)\\
		&\ \ \ \ \ \ \ \ \ \ \ +Q_t+\Theta_t^\top R_t\Theta_t+S_t^\top\Theta_t+\Theta_t^\top S_t\Big)Z_t,Z_t\Big\rangle\\
		&\ \ \ \ \ \ \ +\Big\langle\Big(\dot{\Pi}_t+((\widehat{A}_t+\widehat{B}_t\widehat{\Theta}_t)^\top \Pi_t+\Pi_t(\widehat{A}_t+\widehat{B}_t\widehat{\Theta}_t)+(\widehat{C}_t+\widehat{D}_t\widehat{\Theta}_t)^\top P_t(\widehat{C}_t+\widehat{D}_t\widehat{\Theta}_t)\\		
&\ \ \ \ \ \ \ \ \ \ \ +\widehat{Q}_t+\widehat{\Theta}_t^\top\widehat{R}_t\widehat{\Theta}_t
+\widehat{S}_t^\top\widehat{\Theta}_t+\widehat{\Theta}_t^\top\widehat{S}_t\Big)Y_t,Y_t\Big\rangle\\
		&\ \ \ \ \ \ \ +2\Big\langle Y_t,\Big((\widehat{C}_t+\widehat{D}_t\widehat{\Theta}_t)^\top P_t\widehat{D}_t+(\widehat{S}_t+\widehat{R}_t\widehat{\Theta}_t)^\top+\Pi_t\widehat{B}_t\Big)v_t\Big\rangle\\
		&\ \ \ \ \ \ \ +2\Big\langle Y_t,\dot{\eta}_t+(\widehat{A}_t+\widehat{B}_t\widehat{\Theta}_t)^\top\eta_t+(\widehat{C}_t+\widehat{D}_t\widehat{\Theta}_t)^\top P_t\sigma_t+\Pi_tb_t+q_t+\widehat{\Theta}_t^\top\rho_t\Big\rangle\\
		&\ \ \ \ \ \ \ +\Big\langle(\widehat{R}_t+\widehat{D}_t^\top P_t\widehat{D}_t)v_t,v_t\Big\rangle+2\Big\langle v_t,\widehat{B}_t^\top\eta_t+\widehat{D}_t^\top P_t\sigma_t+\rho_t\Big\rangle\\
		&\ \ \ \ \ \ \ +\Big\langle P_t\sigma_t,\sigma_t\Big\rangle+2\Big\langle\eta_t,b_t\Big\rangle\Big)\d t-\tau\mathcal{E}(x,u^{0,x,\Theta,\overline{\Theta},v}).
	\end{aligned}
\end{equation}}
Here the last equality holds since
{\footnotesize\begin{equation*}
	\begin{aligned}
		\tau\mathcal{E}(x,u^{0,x,\Theta,\overline{\Theta},v})
=&\mathbb{E}\int_0^\tau F(r,X_t,\mathbb{E}X_t,u_t,\mathbb{E}u_t)d t\\
=&\int_{0}^{\tau}\Big(\mathbb{E}\Big\langle(Q_t+S^\top\Theta_t+\Theta_t^\top S_t+\Theta_t^\top R_t\Theta_t)(X_t-\mathbb{E}X_t),(X_t-\mathbb{E}X_t)\Big\rangle\\
&\ \ \ \ \ \ \ +\Big\langle(\widehat{Q}_t+\widehat{S}_t^\top\widehat{\Theta}_t+\widehat{\Theta}_t^\top\widehat{S}_t+\widehat{\Theta}_t^\top\widehat{R}_t\widehat{\Theta}_t)\mathbb{E}X_t,\mathbb{E}X_t\Big\rangle+2\Big\langle(\widehat{S}_t+\widehat{R}_t\widehat{\Theta}_t)\mathbb{E}X_t ,v_t\Big\rangle\\
&\ \ \ \ \ \ \ +2\left\langle\mathbb{E}X_t,q_t+\widehat{\Theta}_t^\top\rho_t\right\rangle+\left\langle\widehat{R}_tv_t,v_t\right\rangle+2\left\langle\rho_t,v_t\right\rangle\Big)\d t.
\end{aligned}
\end{equation*}}
By dividing by $t$ on both sides of (\ref{wz18}) we arrive at the conclusion.
\end{proof}

Besides, referring to \cite{ref30} we have the well-posednesses of periodic solutions to Riccati equations.
\begin{proposition}\label{periodic Riccati equation} \cite[Propostion 5.3]{ref30}  Assume {\rm$\textbf{(A1)}$}--{\rm$\textbf{(A3)}$}. Then the following Riccati equation
\begin{equation}\label{wz21}
	\begin{aligned}
	&\dot{P}_t+Q_t+A_t^\top P_t+P_tA_t+C_t^\top P_tC_t\\
	&-(B_t^\top P_t+D_t^\top P_tC_t+S_t)^\top(R_t+D_t^\top P_tD_t)^{-1}(B_t^\top P_t+D_t^\top P_tC_t+S_t)=0
	\end{aligned}
\end{equation}
has a unique uniformly positive definite solution $P\in \mathscr{D}_\tau(\mathbb{S}^n)$. Moreover,
\begin{equation*}
	\begin{aligned}
		-(R_t+D_t^\top P_tD_t)^{-1}(B_t^\top P_t+D_t^\top P_tC_t+S_t)
	\end{aligned}
\end{equation*}
is a $\tau$-periodic stabilizer of $[A,C;B,D]$.
\end{proposition}

Notice that by \textbf{(A2)} there exists $\alpha>0$ such that for $t\ge0$,
\begin{equation*}
	\begin{aligned}
		\widehat{Q}_t-\widehat{S}_t^\top\widehat{R}_t^{-1}\widehat{S}_t\ge\alpha I_n.
	\end{aligned}
	\nonumber
\end{equation*}
Hence
{\footnotesize\begin{equation*}
	\begin{aligned}
		&\ \ \ \ \widehat{Q}_t+\widehat{C}_t^\top P_t\widehat{C}_t-(\widehat{D}_t^\top P_t\widehat{C}_t+\widehat{S}_t)^\top(\widehat{R}_t+\widehat{D}_t^\top P_t\widehat{D}_t)^{-1}(\widehat{D}_t^\top P_t\widehat{C}_t+\widehat{S}_t)\\	&=\widehat{Q}_t-\widehat{S}_t^\top\widehat{R}_t^{-1}\widehat{S}_t\\
&\ \ +\Big(\widehat{C}_t-\widehat{D}_t(\widehat{R}_t+\widehat{D}_t^\top P_t\widehat{D}_t)^{-1}(\widehat{D}_t^\top P_t\widehat{C}_t+\widehat{S}_t)\Big)^\top P_t\Big(\widehat{C}_t-\widehat{D}_t(\widehat{R}_t+\widehat{D}_t^\top P_t\widehat{D}_t)^{-1}(\widehat{D}_t^\top P_t\widehat{C}_t+\widehat{S}_t)\Big)\\
&\ +\Big(\widehat{S}_t-\widehat{R}_t(\widehat{R}_t+\widehat{D}_t^\top P_t\widehat{D}_t)^{-1}(\widehat{D}_t^\top P_t\widehat{C}_t+\widehat{S}_t)\Big)^\top\widehat{R}_t^{-1}\Big(\widehat{S}_t-\widehat{R}_t(\widehat{R}_t+\widehat{D}_t^\top P_t\widehat{D}_t)^{-1}(\widehat{D}_t^\top P_t\widehat{C}_t+\widehat{S}_t)\Big)\\
		&\ge\alpha I_n.
	\end{aligned}
\end{equation*}}
An immediate corollary of Proposition \ref{periodic Riccati equation} gives the well-posedness of periodic solution to the other Riccati equation relevant to the mean-field term.
\begin{corollary}\label{periodic Riccati equation corollary} Assume {\rm$\textbf{(A1)}$}--{\rm$\textbf{(A3)}$}. Then the following Riccati equation
\begin{equation}\label{wz22}
	\begin{aligned}
	&\dot{\Pi}_t+\widehat{Q}_t+\widehat{A}_t^\top\Pi_t+\Pi_t\widehat{A}_t+\widehat{C}_t^\top P_t\widehat{C}_t\\
	&-(\widehat{B}_t^\top\Pi_t+\widehat{D}_t^\top P_t\widehat{C}_t+\widehat{S}_t)^\top(\widehat{R}_t+\widehat{D}_t^\top P_t\widehat{D}_t)^{-1}(\widehat{B}_t^\top\Pi_t+\widehat{D}_t^\top P_t\widehat{C}_t+\widehat{S}_t)=0
	\end{aligned}
\end{equation}
has a unique uniformly positive definite solution $\Pi\in\mathscr{D}_\tau(\mathbb{S}^n)$. Moreover,
\begin{equation}\label{wz19}
	\begin{aligned}
		-(\widehat{R}_t+\widehat{D}_t^\top P_t\widehat{D}_t)^{-1}(\widehat{B}_t^\top\Pi_t+\widehat{D}_t^\top P_t\widehat{C}_t+\widehat{S}_t)
	\end{aligned}
\end{equation}
 is a $\tau$-periodic stabilizer of $[\widehat{A},0;\widehat{B},0]$.
\end{corollary}

We also have the well-possedness of periodic solution to an ODE relevant to non-homogeneous terms in the state equation and $1$-order terms in the cost functional.
\begin{proposition}\label{wz20} Assume {\rm$\textbf{(A1)}$}--{\rm$\textbf{(A3)}$}. Then the following ODE
\begin{equation}\label{periodic one-order ODE}
	\begin{aligned}
    &\dot{\eta}_t+(\widehat{A}_t-\widehat{B}_t(\widehat{R}_t+\widehat{D}_t^\top P_t\widehat{D}_t)^{-1}(\widehat{B}_t^\top\Pi_t+\widehat{D}_t^\top P_t\widehat{C}_t+\widehat{S}_t))^\top\eta_t\\
    &+(\widehat{C}_t-\widehat{D}_t(\widehat{R}_t+\widehat{D}_t^\top P_t\widehat{D}_t)^{-1}(\widehat{B}_t^\top\Pi_t+\widehat{D}_t^\top P_t\widehat{C}_t+\widehat{S}_t))^\top P_t\sigma_t\\
    &+\Pi_tb_t+q_t-(\widehat{B}_t^\top\Pi_t+\widehat{D}_t^\top P_t\widehat{C}_t+\widehat{S}_t)^\top(\widehat{R}_t+\widehat{D}_t^\top P_t\widehat{D}_t)^{-1}\rho_t=0
    \end{aligned}
\end{equation}
has a unique uniformly positive definite solution $\eta\in\mathscr{D}_\tau(\mathbb{R}^n)$.
\end{proposition}
\begin{proof}
Consider a linear ODE with a terminal time $\tau$
\begin{equation*}
	\left\{\begin{aligned}
		\d\eta_t=&-\Big((\widehat{A}_t-\widehat{B}_t(\widehat{R}_t+\widehat{D}_t^\top P_t\widehat{D}_t)^{-1}(\widehat{B}_t^\top\Pi_t+\widehat{D}_t^\top P_t\widehat{C}_t+\widehat{S}_t))^\top\eta_t\\
		&\ \ \ \ +(\widehat{C}_t-\widehat{D}_t(\widehat{R}_t+\widehat{D}_t^\top P_t\widehat{D}_t)^{-1}(\widehat{B}_t^\top\Pi_t+\widehat{D}_t^\top P_t\widehat{C}_t+\widehat{S}_t))^\top P_t\sigma_t+\Pi_tb_t+q_t\\
		&\ \ \ \ -(\widehat{B}_t^\top\Pi_t+\widehat{D}_t^\top P_t\widehat{C}_t+\widehat{S}_t)^\top(\widehat{R}_t+\widehat{D}_t^\top P_t\widehat{D}_t)^{-1}\rho_t\Big)\d t,\ \ \ t\in[0,\tau],\\
		{\eta}_\tau=&h,
	\end{aligned}\right.
	\nonumber
\end{equation*}
and a linear ODE with initial time $0$
\begin{equation}\label{one-order linear homogeneous system}
	\left\{\begin{aligned}
		\d\psi_t=&\Big(\widehat{A}_t-\widehat{B}_t(\widehat{R}_t+\widehat{D}_t^\top P_t\widehat{D}_t)^{-1}(\widehat{B}_t^\top\Pi_t+\widehat{D}_t^\top P_t\widehat{C}_t+\widehat{S}_t)\Big)\psi_t\d t,\ \ \ t\in[0,\infty),\\
		\psi_0=&I.\\
	\end{aligned}\right.
\end{equation}
By Corollary \ref{periodic Riccati equation corollary}, we know that (\ref{wz19}) is a $\tau$-periodic stabilizer of $[\widehat{A},0;\widehat{B},0]$,
so ODE (\ref{one-order linear homogeneous system}) has a unique solution $\psi\in L^{2}([0,\infty);\mathbb{R}^{n\times n})$ and there exist $M,\lambda>0$ such that for any $t\ge 0$, $|\psi_t|^2\le Me^{-\lambda t}$.

Then a calculus of $\int_t^\tau\d(\psi^\top_s\eta_s)$ leads to
{\small\begin{equation}\label{solution of linear ODE}
	\begin{split}
	\eta_t=&\Big(\psi_\tau\psi_t^{-1}\Big)^\top h\\	&+\int_t^\tau\Big(\psi_s\psi_t^{-1}\Big)^\top\Big((\widehat{C}_s-\widehat{D}_s(\widehat{R}_s+\widehat{D}_s^\top P_s\widehat{D}_s)^{-1}(\widehat{B}_s^\top\Pi_s+\widehat{D}_s^\top P_s\widehat{C}_s+\widehat{S}_s))^\top P_s\sigma_s\\
	&\ \ \ \ \ \ \ \ \ \ \ \ \ \ \ \ \ \ \ \ \ \ \ \ +\Pi_sb_s+\widehat{q}_s-(\widehat{B}_s^\top\Pi_s+\widehat{D}_s^\top P_s\widehat{C}_s+\widehat{S}_s)^\top(\widehat{R}_s+\widehat{D}_s^\top P_s\widehat{D}_s)^{-1}\rho_s\Big)\d s.
\end{split}
\end{equation}}

Consider a matrix sequence $\big\{M_k\big\}_{k\in\mathbb{N}}\subset\mathbb{R}^{n\times n}$, where for $k\in\mathbb{N}$, $M_k=\sum\limits_{i=0}^k\psi_{i\tau}$.
Note that for any $0\le k_1<k_2$,
\begin{equation*}
	\begin{aligned}		\big|M_{k_2}-M_{k_1}\big|^2=&\text{tr}\Big[\sum_{i=k_1+1}^{k_2}\psi_{i\tau}^\top\sum_{i=k_1+1}^{k_2}\psi_{i\tau}\Big]\le\sum_{i=k_1+1}^{k_2}\text{tr}\Big[e^{\frac{\lambda i\tau}{2}}\psi_{i\tau}^\top\psi_{i\tau}\Big]\sum_{i=k_1+1}^{k_2}e^{-\frac{\lambda i\tau}{2}}\\
		\le&K\sum_{i=k_1+1}^{k_2}e^{-\frac{\lambda i\tau}{2}}\sum_{i=k_1+1}^{k_2}e^{-\frac{\lambda i\tau}{2}}\le Ke^{-\lambda(k_1+1)\tau}(1-e^{-\frac{\lambda\tau}{2}})^{-2}.
	\end{aligned}
\end{equation*}
Thus $\big\{M_k\big\}_{k\in\mathbb{N}}$ is a Cauchy sequence in $\mathbb{R}^{n\times n}$, and there is $M\in\mathbb{R}^{n\times n}$ such that $\lim\limits_{k\rightarrow\infty}|M_k-M|=0$. Using the fact that (\ref{wz19})
is a $\tau$-periodic stabilizer of $[\widehat{A},0;\widehat{B},0]$ again we know that both $\psi_{\cdot+\tau}\psi_\tau^{-1}$ and $\psi_{\cdot}$ satisfy \eqref{one-order linear homogeneous system} which only has a unique solution. So $\psi_{\cdot+\tau}\psi_\tau^{-1}=\psi_{\cdot}$ and we have
\begin{equation}\label{sum flow property}
	\begin{split}		M_k(I-\psi_\tau)&=\sum_{i=0}^k\psi_{i\tau}(I-\psi_\tau)=\sum_{i=0}^k\psi_{i\tau}-\sum_{i=1}^{k+1}\psi_{i\tau}=I-\psi_{(k+1)\tau}.
	\end{split}
\end{equation}
Recalling that
$|\psi_{(k+1)\tau}|^2\le Ke^{-\lambda (k+1)\tau}$, as $k\rightarrow\infty$ in \eqref{sum flow property}, we have $M(I-\psi_\tau)=I$. Hence $(I-\psi_\tau)$ is an invertible matrix.

Set
\begin{equation*}
	\begin{aligned}
		l=\int_0^\tau\psi_s^\top\Big(&(\widehat{C}_s-\widehat{D}_s(\widehat{R}_s+\widehat{D}_s^\top P_s\widehat{D}_s)^{-1}(\widehat{B}_s^\top\Pi_s+\widehat{D}_s^\top P_s\widehat{C}_s+\widehat{S}_s))^\top P_s\sigma_s\\
&+\Pi_sb_s+\widehat{q}_s-(\widehat{B}_s^\top\Pi_s+\widehat{D}_s^\top P_s\widehat{C}_s+\widehat{S}_s)^\top(\widehat{R}_s+\widehat{D}_s^\top P_s\widehat{D}_s)^{-1}\rho_s\Big)\d s.
	\end{aligned}
\end{equation*}
We take
\begin{equation*}
	\begin{aligned}
		&h=\Big((I-\psi_\tau)^{-1}\Big)^\top l.
	\end{aligned}
\end{equation*}
Then
\begin{equation*}
	\begin{aligned}
		&h=\psi_\tau^\top h+l,
	\end{aligned}
	\nonumber
\end{equation*}
which, together with \eqref{solution of linear ODE}, leads to
\begin{equation*}
	\begin{aligned}
		\eta_0=\eta_\tau=h.
	\end{aligned}
\end{equation*}
Consider the infinite horizon ODE (\ref{periodic one-order ODE}) on $[0,\infty)$ with initial value $\eta_0=h$.
Bearing in mind that the coefficients of (\ref{periodic one-order ODE}) is periodic and (\ref{wz19})
is a $\tau$-periodic stabilizer of $[\widehat{A},0;\widehat{B},0]$, we know that this equation with initial value has a unique solution, and it is easy to verify that both $\eta_\cdot$ and $\eta_{\cdot+\tau}$ are its solutions. Hence $\eta_\cdot=\eta_{\cdot+\tau}$ and $\eta$ is a periodic solution with $\eta_0=h$.

As for the uniqueness of solution, if there exists two periodic solutions to ODE (\ref{periodic one-order ODE}) with the initial values $h$ and $h'$, respectively. Denote $\delta h=h-h'$, and
then
\begin{equation*}
	\begin{aligned}
		(I-\psi_\tau)^\top\delta h=0.
	\end{aligned}
\end{equation*}
Since $(I-\psi_\tau)$ is invertible, we have $\delta h=0$.
\end{proof}

Now we are ready to give an explicit optimal control of \textbf{Problem (CL-MFLQE)}.
\begin{theorem}\label{wz29} Assume {\rm$\textbf{(A1)}$}--{\rm$\textbf{(A3)}$},
and $(\Pi,P,\eta)\in\mathscr{D}_\tau(\mathbb{S}^n)\times\mathscr{D}_\tau(\mathbb{S}^n)\times\mathscr{D}_\tau(\mathbb{R}^n)$ are the solutions to (\ref{wz21}), (\ref{wz22}) and (\ref{periodic one-order ODE}), respectively.
{\rm\textbf{Problem (CL-MFLQE)}} is closed-loop solvable with the optimal controls
\begin{equation*}
	\begin{aligned}
		u^{0,x,\Theta^*,\overline{\Theta}^*,v^*}_t=\Theta^*_t X^{0,x,\Theta^*,\overline{\Theta}^*,v^*}_t
+\overline{\Theta}^*_t\mathbb{E}X^{0,x,\Theta^*,\overline{\Theta}^*,v^*}_t+v^*_t,
	\end{aligned}
\end{equation*}
where $(\Theta^*,\overline{\Theta}^*,v^*)\in\mathbb{U}$ with the corresponding periodic measure $\mu^{\Theta^*,\overline{\Theta}^*,v^*}$ defined in Definition \ref{wz23} satisfies for $t\geq0$,
\begin{equation}\label{wz24}
	\begin{aligned}
\int_{\mathbb{R}^n}&\big(y-\int_{\mathbb{R}^n}z\mu_t^{\Theta^*,\overline{\Theta}^*,v^*}(\d z)\big)^\top\big(\Theta^*_t-\Theta^0_t\big)^\top\\
&\times\big(\Theta^*_t-\Theta^0_t\big)\big(y-\int_{\mathbb{R}^n}z\mu_t^{\Theta^*,\overline{\Theta}^*,v^*}(\d z)\big)\mu_t^{\Theta^*,\overline{\Theta}^*,v^*}(\d y)=0\\
	\end{aligned}
\end{equation}
and	
\begin{equation}\label{wz25}
\begin{aligned}
v^*_t-v^0_t+\big(\widehat{\Theta}^*_t-\widehat{\Theta}^0_t\big)\int_{\mathbb{R}^n}z\mu_t^{\Theta^*,\overline{\Theta}^*,v^*}(\d z)=0,
\end{aligned}
\end{equation}
with
\begin{equation*}
	\begin{aligned}
		\Theta^0_t=&-(R_t+D_t^\top P_tD_t)^{-1}(B_t^\top P_t+D_t^\top P_tC_t+S_t),\\
		\widehat{\Theta}^0_t=&-(\widehat{R}_t+\widehat{D}_t^\top P_t\widehat{D}_t)^{-1}(\widehat{B}_t^\top\Pi_t+\widehat{D}_t^\top P_t\widehat{C}_t+\widehat{S}_t),\\
		v^0_t=&-(\widehat{R}_t+\widehat{D}_t^\top P_t\widehat{D}_t)^{-1}(\widehat{B}_t^\top\eta_t+\widehat{D}_t^\top P_t\sigma_t+\rho_t).
	\end{aligned}
\end{equation*}
In particular, an optimal control is
\begin{equation}\label{wz28}
	\begin{aligned}		
u^{0,x,\Theta^0,\overline{\Theta}^0,v^0}_t=\Theta^0_t X^{0,x,\Theta^0,\overline{\Theta}^0,v^0}_t
+\overline{\Theta}^0_t\mathbb{E}X^{0,x,\Theta^0,\overline{\Theta}^0,v^0}_t+v^0_t,
	\end{aligned}
\end{equation}
and the value function in this case is presented by
{\small\begin{equation}\label{wz27}
	\begin{aligned}
	V(x)=\frac{1}{\tau}\int_0^\tau\Big(&-\Big\langle(\widehat{R}_t+\widehat{D}_t^\top P_t\widehat{D}_t)^{-1}(\widehat{B}_t^\top\eta_t+\widehat{D}_t^\top P_t\sigma_t+\rho_t),(\widehat{B}_t^\top\eta_t+\widehat{D}_t^\top P_t\sigma_t+\rho_t)\Big\rangle\\
&+\left\langle P_t\sigma_t,\sigma_t\right\rangle+2\left\langle\eta_t,b_t\right\rangle\Big)\d t.
    \end{aligned}
\end{equation}}
\end{theorem}
\begin{proof}
Take $\xi$ independent of $\mathbb{F}$ satisfying $\mathbb{P}_{\xi}=\mu_0^{\Theta,\overline{\Theta},v}$ and write $X^{0,\xi,\Theta,\overline{\Theta},v}_\cdot$ as $X_\cdot$ in this proof to save space.
For $t\geq0$, denote
\begin{equation*}
	\begin{aligned}
		Y_t=\mathbb{E}X_t,\ \ \ Z_t=X_t-Y_t,\ \ \ R^P_t=R_t+D_t^\top P_t D_t.
	\end{aligned}
	\nonumber
\end{equation*}
First for $(\Theta,\overline{\Theta},v)\in\mathbb{U}$, by Proposition \ref{cost functional coefficients} we have
{\begin{equation*}
	\begin{aligned}
		Q^{\Theta,P}_t=&(A_t+B_t\Theta_t)^\top P_t+P_t(A_t+B_t\Theta_t)+Q_t+\Theta_t^\top R_t\Theta_t\\
		&+(C_t+D_t\Theta_t)^\top P_t(C_t+D_t\Theta_t)+S_t^\top\Theta_t+\Theta_t^\top S_t\\
		=&Q_t+A_t^\top P_t+P_tA_t+C_t^\top P_tC_t\\
&-(B_t^\top P_t+D_t^\top P_tC_t+S_t)^\top(R_t+D_t^\top P_tD_t)^{-1}(B_t^\top P_t+D_t^\top P_tC_t+S_t)\\
		&+(\Theta_t-\Theta^0_t)^\top(R_t+D_t^\top P_tD_t)(\Theta_t-\Theta^0_t)\\
		=&Q^{\Theta^0,P}_t+(\Theta_t-\Theta^0_t)^\top R^P_t(\Theta_t-\Theta^0_t),\\
Q^{\widehat{\Theta},\Pi,P}_t=&(\widehat{A}_t+\widehat{B}_t\widehat{\Theta}_t)^\top \Pi_t+\Pi_t(\widehat{A}_t+\widehat{B}_t\widehat{\Theta}_t)+\widehat{Q}_t+\widehat{\Theta}_t^\top\widehat{R}_t\widehat{\Theta}_t\\
		&+(\widehat{C}_t+\widehat{D}_t\widehat{\Theta}_t)^\top P_t(\widehat{C}_t+\widehat{D}_t\widehat{\Theta}_t)+\widehat{S}_t^\top\widehat{\Theta}_t+\widehat{\Theta}_t^\top\widehat{S}_t\\
		=&\widehat{Q}_t+\widehat{A}_t^\top\Pi_t+\Pi_t\widehat{A}_t+\widehat{C}_t^\top P_t\widehat{C}_t\\
		&-(\widehat{B}_t^\top\Pi_t+\widehat{D}_t^\top P_t\widehat{C}_t+\widehat{S}_t)^\top(\widehat{R}_t+\widehat{D}_t^\top P_t\widehat{D}_t)^{-1}(\widehat{B}_t^\top\Pi_t+\widehat{D}_t^\top P_t\widehat{C}_t+\widehat{S}_t)\\
		&+(\widehat{\Theta}_t-\widehat{\Theta}^0_t)^\top(\widehat{R}_t+\widehat{D}_t^\top P_t\widehat{D}_t)(\widehat{\Theta}_t-\widehat{\Theta}^0_t)\\		=&Q^{\widehat{\Theta}^0,\Pi,P}_t+(\widehat{\Theta}_t-\widehat{\Theta}^0_t)^\top\widehat{R}^P_t(\widehat{\Theta}_t-\widehat{\Theta}^0_t),\\	
K^{\widehat{\Theta},\Pi,P}_t=&(\widehat{C}_t+\widehat{D}_t\widehat{\Theta}_t)^\top P_t\widehat{D}_t+(\widehat{S}_t+\widehat{R}_t\widehat{\Theta}_t)^\top+\Pi_t\widehat{B}_t\\
		=&(\widehat{\Theta}_t-\widehat{\Theta}^0_t)^\top(\widehat{R}_t+\widehat{D}_t^\top P_t\widehat{D}_t)\\
		=&(\widehat{\Theta}_t-\widehat{\Theta}^0_t)^\top\widehat{R}^P_t,\\	
L^{\widehat{\Theta},\Pi,P,\eta}_t=&(\widehat{A}_t+\widehat{B}_t\widehat{\Theta}_t)^\top\eta_t+(\widehat{C}_t+\widehat{D}_t\widehat{\Theta}_t)^\top P_t\sigma_t+\Pi_tb_t+q_t+\widehat{\Theta}_t^\top\rho_t\\
		=&(\widehat{A}_t-\widehat{B}_t(\widehat{R}_t+\widehat{D}_t^\top P_t\widehat{D}_t)^{-1}(\widehat{B}_t^\top\Pi_t+\widehat{D}_t^\top P_t\widehat{C}_t+\widehat{S}_t))^\top\eta_t\\
&+(\widehat{C}_t-\widehat{D}_t(\widehat{R}_t+\widehat{D}_t^\top P_t\widehat{D}_t)^{-1}(\widehat{B}_t^\top\Pi_t+\widehat{D}_t^\top P_t\widehat{C}_t+\widehat{S}_t))^\top P_t\sigma_t+\Pi_tb_t+q_t\\
&-(\widehat{B}_t^\top\Pi_t+\widehat{D}_t^\top P_t\widehat{C}_t+\widehat{S}_t)^\top(\widehat{R}_t+\widehat{D}_t^\top P_t\widehat{D}_t)^{-1}\rho_t\\
&+(\widehat{\Theta}_t-\widehat{\Theta}^0_t)^\top(\widehat{B}_t^\top\eta_t+\widehat{D}_t^\top P_t\sigma_t+\rho_t)\\
		=&L^{\widehat{\Theta}^0,\Pi,P,\eta}_t+(\widehat{\Theta}_t-\widehat{\Theta}^0_t)^\top l^{P,\eta}_t.
	\end{aligned}
\end{equation*}}
Then it follows from Proposition \ref{transffered form by ODE} again that
{\small\begin{equation}\label{wz26}
	\begin{aligned}
		&\mathcal{E}(x,u^{0,x,\Theta,\overline{\Theta},v})\\
		=&\frac{1}{\tau}\int_{0}^{\tau}\Big(\mathbb{E}\left\langle (\dot{P}_t+Q^{\Theta,P}_t)Z_t,Z_t\right\rangle+\left\langle (\dot{\Pi}_t+Q^{\widehat{\Theta},\Pi,P}_t)Y_t,Y_t\right\rangle+2\left\langle Y_t,K^{\widehat{\Theta},\Pi,P}_tv_t\right\rangle\\
&\ \ \ \ \ \ \ \ \ +2\left\langle Y_t,\dot{\eta}_t+L^{\widehat{\Theta},\Pi,P,\eta}_t\right\rangle+\left\langle \widehat{R}^P_tv_t,v_t\right\rangle+2\left\langle v_t,l^{P,\eta}_t\right\rangle+\left\langle P_t\sigma_t,\sigma_t\right\rangle+2\left\langle\eta_t,b_t\right\rangle\Big)\d r\\		=&\frac{1}{\tau}\int_0^\tau\Big(\mathbb{E}\left\langle(\dot{P}_t+Q^{\Theta^0,P}_t)Z_t,Z_t\right\rangle+\mathbb{E}\left\langle R^P_t(\Theta_t-\Theta^0_t)Z_t,(\Theta_t-\Theta^0_t)Z_t\right\rangle\\
&\ \ \ \ \ \ \ \ \ +\left\langle (\dot{\Pi}_t+Q^{\widehat{\Theta}^0,\Pi,P}_t)Y_t,Y_t\right\rangle+\left\langle\widehat{R}^P_t(\widehat{\Theta}_t-\widehat{\Theta}^0_t)Y_t,(\widehat{\Theta}_t-\widehat{\Theta}^0_t)Y_t\right\rangle\\ &\ \ \ \ \ \ \ \ \ +2\left\langle\widehat{R}^P_t(\widehat{\Theta}_t-\widehat{\Theta}^0_t)Y_t,v_t\right\rangle+2\left\langle Y_t,\dot{\eta}_t+L^{\widehat{\Theta}^0,\Pi,P,\eta}_t\right\rangle+2\left\langle(\widehat{\Theta}_t-\widehat{\Theta}^0_t)Y_t,l^{P,\eta}_t\right\rangle\\
		&\ \ \ \ \ \ \ \ \ +\left\langle \widehat{R}^P_tv_t,v_t\right\rangle+2\left\langle v_t,l^{P,\eta}_t\right\rangle+\left\langle P_t\sigma_t,\sigma_t\right\rangle+2\left\langle\eta_t,b_t\right\rangle\Big)\d r\\
		=&\frac{1}{\tau}\int_0^\tau\Big(\mathbb{E}\left\langle R^P_t(\Theta_t-\Theta^0_t)Z_t,(\Theta_t-\Theta^0_t)Z_t\right\rangle\\		
&\ \ \ \ \ \ \ \ \ +\Big\langle\widehat{R}^P_t\left((\widehat{\Theta}_t-\widehat{\Theta}^0_t)Y_t+v_t
+(\widehat{R}^P_t)^{-1}l^{P,\eta}_t\right),\left((\widehat{\Theta}_t-\widehat{\Theta}^0_t)Y_t+v_t+(\widehat{R}^P_t)^{-1}l^{P,\eta}_t\right)\Big\rangle\\
&\ \ \ \ \ \ \ \ \ -\left\langle(\widehat{R}^P_t)^{-1}l^{P,\eta}_t,l^{P,\eta}_t\right\rangle+\left\langle P_t\sigma_t,\sigma_t\right\rangle+2\left\langle\eta_t,b_t\right\rangle\Big)\d r\\	\ge&\frac{1}{\tau}\int_0^\tau\Big(-\left\langle(\widehat{R}^P_t)^{-1}l^{P,\eta}_t,l^{P,\eta}_t\right\rangle+\left\langle P_t\sigma_t,\sigma_t\right\rangle+2\left\langle\eta_t,b_t\right\rangle\Big)\d r\\
		=&\frac{1}{\tau}\int_0^\tau\Big(-\Big\langle(\widehat{R}_t+\widehat{D}_t^\top P_t\widehat{D}_t)^{-1}(\widehat{B}_t^\top\eta_t+\widehat{D}_t^\top P_t\sigma_t+\rho_t),(\widehat{B}_t^\top\eta_t+\widehat{D}_t^\top P_t\sigma_t+\rho_t)\Big\rangle\\
&\ \ \ \ \ \ \ \ \ +\left\langle P_t\sigma_t,\sigma_t\right\rangle+2\left\langle\eta_t,b_t\right\rangle\Big)\d r.
	\end{aligned}
\end{equation}}
Obviously, the equality holds in above if and only if $(\Theta^*,\overline{\Theta}^*,v^*)\in\mathbb{U}$ satisfies for $t\in[0,\tau)$,
\begin{equation*}
	\begin{aligned}
		\mathbb{E}[Z_t^\top(\Theta^*_t-\Theta^0_t)^\top(\Theta^*_t-\Theta^0_t)Z_t]=0\ \ {\rm and}\ \
		(\widehat{\Theta}^*_t-\widehat{\Theta}^0_t)Y_t+v^*_t+(\widehat{R}^P_t)^{-1}l^{P,\eta}_t=0,
	\end{aligned}
\end{equation*}
which implies (\ref{wz24}) and (\ref{wz25}) and gives the expression of the closed-loop optimal controls (\ref{wz28}) for $t\in[0,\tau)$.

In particular,
\begin{equation*}
	\begin{aligned}
		(\Theta^*,\overline{\Theta}^*,v^*)=(\Theta^0,\overline{\Theta}^0,v^0)
	\end{aligned}
\end{equation*}
is an optimal control since $(\Theta^0,\overline{\Theta}^0,v^0)\in\mathbb{U}$ by Proposition \ref{periodic Riccati equation} and Corollary \ref{periodic Riccati equation corollary}. With $(\Theta^0,\overline{\Theta}^0,v^0)$, the minimum value in (\ref{wz26}) is gotten, which implies (\ref{wz27}).

For any $t=k\tau+r$, where $r\in[0,\tau)$, $k\in\mathbb{N}$. Since $(\Theta,\overline{\Theta},v)$, $\mu^{\Theta^*,\overline{\Theta}^*,v^*}$, $(\Pi,P,\eta)$ and coefficients in \textbf{Problem (CL-MFLQE)} appearing in above argument are all $\tau$-periodic, all the results obtained for $t\in[0,\tau)$ still holds for $t\in[0,\infty)$.
\end{proof}

\begin{rmk} We observe from the value function (\ref{wz27}) that if non-homogeneous terms in the closed-loop state equation and 1-order terms in the cost function vanish, the value function will be $0$. This phenomenon exactly coincides with the fact that the square expectation of the state tends to $0$ if the closed-loop state equation does not involve non-homogeneous terms, and consequently the cost functional without 1-order terms also tends to $0$.
\end{rmk}

\section{Examples}

Finally, we present one example to illustrate our theoretical results.
\begin{example} Consider a 2-dimensional SDE
{\begin{equation*}
	\left\{\begin{aligned}
		\d X_t=&\Bigg(\begin{pmatrix}-1&\cos t\\0&-1\end{pmatrix}X_t+\begin{pmatrix}0&\sin t-\cos t\\0&0\end{pmatrix}\mathbb{E}X_t+\begin{pmatrix}1&0\\0&1\end{pmatrix}u_t+\begin{pmatrix}\sin t\\1\end{pmatrix}\Bigg)\d t\\
		&+\Bigg(\begin{pmatrix}2\cos t&0\\0&2\cos t\end{pmatrix}X_t+\begin{pmatrix}\sin t-2\cos t&0\\0&\sin t-2\cos t\end{pmatrix}\mathbb{E}X_t\\
		&+\begin{pmatrix}\cos t\\1\end{pmatrix}\Bigg)\d W_t,\ \ \ t\ge 0,\\
		X_0=&x.
	\end{aligned}\right.
\end{equation*}}
with the ergodic quadratic cost functional
{\begin{equation*}
	\begin{aligned}
		&F(t,X_t,\mathbb{E}X_t,u_t,\mathbb{E}u_t)\\
		=&\left\langle\begin{pmatrix}35-19\cos^2t&\sin t+7\cos t-4\cos^3t\\\sin t+7\cos t-4\cos^3t&35-21\cos^2t\end{pmatrix}X_t,X_t\right\rangle+\left\langle u_t,u_t\right\rangle\\
		&+2\left\langle\begin{pmatrix}-3\sin t+\cos t-6\sin t\cos t\\-3-5\sin t-\cos t-\sin t\cot t-\sin^2 t+\sin ^3t\end{pmatrix},X_t\right\rangle-2\left\langle\begin{pmatrix}\cos t\\\sin t\end{pmatrix},u_t\right\rangle\\
		&+\left\langle\begin{pmatrix}-24+23\cos^2t&4\sin t-9\cos t+5\cos^3t\\4\sin t-9\cos t+5\cos^3t&-26+27\cos^2t\end{pmatrix}\mathbb{E}X_t,\mathbb{E}X_t\right\rangle.
	\end{aligned}
\end{equation*}}
\end{example}
Here the corresponding coefficients in this example are
{\begin{equation*}
	\begin{aligned}
		&A_t=\begin{pmatrix}-1&\cos t\\0&-1\end{pmatrix},\ B_t=I_2,\ C_t=2\cos tI_2,\ D_t=0,\ b_t=\begin{pmatrix}\sin t\\1\end{pmatrix},\\
		&\widehat{A}_t=\begin{pmatrix}-1&\sin t\\0&-1\end{pmatrix},\ \widehat{B}_t=I_2,\ \widehat{C}_t=\sin tI_2,\ \widehat{D}_t=0,\ \sigma_t=\begin{pmatrix}\cos t\\1\end{pmatrix},\\
		&Q_t=\begin{pmatrix}35-19\cos^2t&\sin t+7\cos t-4\cos^3t\\\sin t+7\cos t-4\cos^3t&35-21\cos^2t\end{pmatrix},\ R_t=I_2,\ S_t=0,\\
		&\widehat{Q}_t=\begin{pmatrix}15-4\sin^2t&5\sin t-2\cos t+\cos^3t\\5\sin t-2\cos t+\cos^3t&15-6\sin^2t\end{pmatrix},\ \widehat{R}_t=I_2,\ \widehat{S}_t=0,\\
		&q_t=\begin{pmatrix}-3\sin t+\cos t-6\sin t\cos t\\-3-5\sin t-\cos t-\sin t\cot t-\sin^2 t+\sin ^3t\end{pmatrix},\ \rho_t=-\begin{pmatrix}\cos t\\\sin t\end{pmatrix},
	\end{aligned}
\end{equation*}}
which are all bounded and periodic with a period $2\pi$. Hence \textbf{(A1)} holds.
As for \textbf{(A2)}, it is satisfied since
{\begin{equation*}
	\begin{aligned}
		Q_t-S_t^\top R_t^{-1}S_t=&\begin{pmatrix}35-19\cos^2t&\sin t+7\cos t-4\cos^3t\\\sin t+7\cos t-4\cos^3t&35-21\cos^2t\end{pmatrix}\ge 2I_2,\\
		\widehat{Q}_t-\widehat{S}_t^\top\widehat{R}_t^{-1}\widehat{S}_t=&\begin{pmatrix}15-4\sin^2t&5\sin t-2\cos t+\cos^3t\\5\sin t-2\cos t+\cos^3t&15-6\sin^2t\end{pmatrix}\ge I_2.
	\end{aligned}
\end{equation*}}
To verify \textbf{(A3)}, Take
$\Theta(\cdot)=-I_2\in \mathscr{B}_{2\pi}(\mathbb{R}^{2\times2})$, and then the homogeneous system
{\begin{equation*}
	\left\{\begin{aligned}
		\d\Phi_t&=\left(\begin{pmatrix}-1&\cos t\\0&-1\end{pmatrix}-\begin{pmatrix}1&0\\0&1\end{pmatrix}\right)\Phi_t\d t+\begin{pmatrix}2\cos t&0\\0&2\cos t\end{pmatrix}\Phi_t\d W_t,\ \ \ t\ge 0,\\
		\Phi_0&=I_2,
	\end{aligned}\right.
\end{equation*}}
has a unique solution
{\begin{equation*}
	\begin{aligned}
		\Phi_t={\rm e}^{\int_0^t\Big(-2-2\cos^2s\Big)\d s+\int_0^t2\cos s\d W_s}\begin{pmatrix}1&\sin t\\0&1\end{pmatrix}.
	\end{aligned}
\end{equation*}}
Then for any $t\ge 0$,
\begin{equation*}
	\begin{aligned}
		\mathbb{E}|\Phi_t|^2=&{\rm e}^{\int_0^t\Big(-4+4\cos^2s\Big)\d s}(2+\sin^2 t)={\rm e}^{\int_0^t\Big(2\cos 2s-2\Big)\d s}(2+\sin^2 t)\le 6{\rm e}^{1-2t}.
	\end{aligned}
\end{equation*}
Take
$\widehat{\Theta}(\cdot)=-I_2\in\mathscr{B}_{2\pi}(\mathbb{R}^{2\times2})$, and then the homogeneous system
\begin{equation*}
	\left\{\begin{aligned}
		\d\Psi_t&=\left(\begin{pmatrix}-1&\sin t\\0&-1\end{pmatrix}-\begin{pmatrix}1&0\\0&1\end{pmatrix}\right)\Psi_t\d t,\ t\ge 0,\\
		\Psi_0&=I_2,
	\end{aligned}\right.
\end{equation*}
has a unique solution
\begin{equation*}
	\begin{aligned}
		\Psi_t={\rm e}^{-2t}\begin{pmatrix}1&1-\cos t\\0&1\end{pmatrix}.
	\end{aligned}
\end{equation*}
Then for any $t\ge 0$,
\begin{equation*}
	\begin{aligned}
		|\Psi_t|^2={\rm e}^{-4t}(2+(1-\cos t)^2)\le 6{\rm e}^{-4t}.
	\end{aligned}
\end{equation*}
Moreover, the related Riccati equations and ODE are as follows
\begin{equation*}
	\begin{aligned}
		&\dot{P}_t+\begin{pmatrix}35-19\cos^2t&\sin t+7\cos t-4\cos^3t\\\sin t+7\cos t-4\cos^3t&35-21\cos^2t\end{pmatrix}+\begin{pmatrix}-1&0\\\cos t&-1\end{pmatrix}P_t\\
		&+P_t\begin{pmatrix}-1&\cos t\\0&-1\end{pmatrix}+\begin{pmatrix}4\cos^2 t&0\\0&4\cos^2 t\end{pmatrix}P_t-P_t^\top P_t=0,\\
		&\dot{\Pi}_t+\begin{pmatrix}15-4\sin^2t&5\sin t-2\cos t+\cos^3t\\5\sin t-2\cos t+\cos^3t&15-6\sin^2t\end{pmatrix}+\begin{pmatrix}-1&0\\\sin t&-1\end{pmatrix}\Pi_t\\
		&+\Pi_t\begin{pmatrix}-1&\sin t\\0&-1\end{pmatrix}+\begin{pmatrix}\sin^2 t&0\\0&\sin^2 t\end{pmatrix}P_t-\Pi_t^\top\Pi_t=0,
\end{aligned}
\end{equation*}
\begin{equation*}
	\begin{aligned}			
&\dot{\eta}_t+\Bigg(\begin{pmatrix}-1&0\\ \sin t &-1\end{pmatrix}-\Pi^\top_t\Bigg){\eta}_t+P_t\begin{pmatrix}\sin t\cos t\\\sin t\end{pmatrix}+\Pi_t\begin{pmatrix}\sin t\\1\end{pmatrix}\\
		&+\begin{pmatrix}-3\sin t+\cos t-6\sin t\cos t\\-3-5\sin t-\cos t-\sin t\cot t-\sin^2 t+\sin ^3t\end{pmatrix}+\Pi_t^\top\begin{pmatrix}\cos t\\\sin t\end{pmatrix}=0,
	\end{aligned}
	\nonumber
\end{equation*}
with the solutions
\begin{equation*}
	\begin{aligned}
		P_t=\begin{pmatrix}5&\cos t\\\cos t&5\end{pmatrix},\ \Pi_t=\begin{pmatrix}3&\sin t\\\sin t&3\end{pmatrix},\ \eta_t=\begin{pmatrix}\cos t\\\sin t\end{pmatrix}.
	\end{aligned}
\end{equation*}
By Theorem \ref{wz29} we have a closed-loop optimal control
\begin{equation*}
	\begin{aligned}
		\Theta^0_t=-\begin{pmatrix}5&\cos t\\\cos t&5\end{pmatrix},\ \widehat{\Theta}^0_t=-\begin{pmatrix}3&\sin t\\\sin t&3\end{pmatrix},\ v^0_t=0,
	\end{aligned}
\end{equation*}
and the value function is
{\small\begin{equation*}
	\begin{aligned}
		V(x)=\frac{1}{2\pi}\int_0^{2\pi}\Bigg(\left\langle\begin{pmatrix}5&\cos t\\\cos t&5\end{pmatrix}\begin{pmatrix}\cos t\\1\end{pmatrix},\begin{pmatrix}\cos t\\1\end{pmatrix}\right\rangle+2\left\langle\begin{pmatrix}\cos t\\\sin t\end{pmatrix},\begin{pmatrix}\sin t\\1\end{pmatrix}\right\rangle\Bigg)\d t
		=\frac{17}{2}.
	\end{aligned}
\end{equation*}}

\end{document}